\documentclass[12pt]{amsart}
\usepackage{}
\usepackage{amsmath}
\usepackage{amsfonts}
\usepackage{amssymb}
\usepackage[all]{xy}           %xypic macro for latex2.09
\usepackage{bm}
\usepackage{bbding}
\usepackage{txfonts}
\usepackage{amscd}
\usepackage{xspace}
\usepackage[shortlabels]{enumitem}
\usepackage{ifpdf}

\ifpdf
  \usepackage[colorlinks,final,backref=page,hyperindex]{hyperref}
\else
  \usepackage[colorlinks,final,backref=page,hyperindex,hypertex]{hyperref}
\fi
\usepackage{tikz}
\usepackage[active]{srcltx}

\makeatletter

\newtheorem{thm}{Theorem}[section]
\newtheorem{lem}[thm]{Lemma}
\newtheorem{cor}[thm]{Corollary}
\newtheorem{pro}[thm]{Proposition}
\newtheorem{ex}[thm]{Example}
\theoremstyle{definition}
\newtheorem{rmk}[thm]{Remark}
\newtheorem{defi}[thm]{Definition}

\newcommand{\nc}{\newcommand}
\newcommand{\delete}[1]{}

\nc{\mlabel}[1]{\label{#1}}  % Use this to suppress names
\nc{\mcite}[1]{\cite{#1}}  % Use this to suppress names
\nc{\mref}[1]{\ref{#1}}  % Use this to suppress names
\nc{\mbibitem}[1]{\bibitem{#1}} % Use this to show number
\delete{% Use the next two lines to show names
\nc{\mlabel}[1]{\label{#1}{\hfill \hspace{1cm}{\bf{{\ }\hfill(#1)}}}}
\nc{\mcite}[1]{\cite{#1}{{\bf{{\ }(#1)}}}}  % Use this lines to show names
\nc{\mref}[1]{\ref{#1}{{\bf{{\ }(#1)}}}}  % Use this lines to show names
\nc{\mbibitem}[1]{\bibitem[\bf #1]{#1}} % Use this to show name
}

\newcommand {\emptycomment}[1]{}

\delete{
\setlength{\baselineskip}{1.8\baselineskip}

\newcommand{\emptycomment}[1]{}

}

\nc{\calo}{\mathcal{O}}
\nc{\oop}{$\mathcal{O}$-operator\xspace}
\nc{\oops}{$\mathcal{O}$-operators\xspace}
\nc{\mrho}{{\bm{\varrho}}}
\nc{\bfk}{\mathbf{K}}
\nc{\invlim}{\displaystyle{\lim_{\longleftarrow}}\,}
\nc{\ot}{\otimes}
\nc{\CV}{\mathbf{C}}
\nc{\CLV}{\mathbf{CL}}

\nc{\Oprn}{\Theta}

\newcommand{\Sym}{\mathsf{S}}

\newcommand{\CE}{\mathsf{CE}}

\newcommand{\NR}{\mathsf{NR}}

\newcommand{\lon }{\,\rightarrow\,}
\newcommand{\be }{\begin{equation}}
\newcommand{\ee }{\end{equation}}

\newcommand{\g}{\mathfrak g}
\newcommand{\h}{\mathfrak h}

\newcommand{\huaV}{\mathcal{V}}
\newcommand{\huaW}{\mathcal{W}}
\newcommand{\huaX}{\mathcal{X}}

\newcommand{\huaJ}{\mathcal{J}}

\newcommand{\huaT}{\mathcal{T}}

\newcommand{\frkg}{\mathfrak g}
\newcommand{\frkh}{\mathfrak h}

\newcommand{\frkC}{\mathfrak C}

\newcommand{\frkX}{\mathfrak X}

\newcommand{\Courant}[1]{\left\llbracket  #1\right\rrbracket }

\newcommand{\Id}{{\rm{Id}}}

\newcommand{\br}[1]{   [ \cdot,    \cdot  ]   }

\newcommand{\dM}{\mathrm{d}}

\newcommand{\Hom}{\mathrm{Hom}}
\newcommand{\Der}{\mathrm{Der}}

\newcommand{\gl}{\mathfrak {gl}}

\newcommand{\ad}{\mathrm{ad}}

\newcommand{\Img}{\mathrm{Im}}

\newcommand{\sgn}{\mathrm{sgn}}

\newcommand{\B}{\mathsf{B}}

\newcommand{\LTS}{\mathsf{LTS}}

\newcommand{\Ten}{\mathsf{T}}
\newcommand{\la}{\mathfrak G}
\nc{\oprn}{\theta}

\newcommand{\TTHL}{\rm{\bf $2$TermHomLTS }}
\newcommand{\LTTS}{\rm{\bf LieTri$2$Sys }}

\begin{document}

\title{Cohomology and Homotopy of Lie triple systems}

\author{Haobo Xia}
\address{Department of Mathematics, Jilin University, Changchun 130012, Jilin, China}
\email{xiahb21@mails.jlu.edu.cn}

\author{Yunhe Sheng}
\address{Department of Mathematics, Jilin University, Changchun 130012, Jilin, China}
\email{shengyh@jlu.edu.cn}

\author{Rong Tang}
\address{Department of Mathematics, Jilin University, Changchun 130012, Jilin, China}
\email{tangrong@jlu.edu.cn}

%\date{\today}

\begin{abstract}
In this paper, first we give the controlling algebra of Lie triple systems. In particular, the cohomology of Lie triple systems can be characterized by the controlling algebra. Then using  controlling algebras,  we introduce the notions of homotopy Nambu algebras and  homotopy Lie triple systems.  We show that $2$-term homotopy Lie triple systems is equivalent to  Lie triple $2$-systems, and the latter is the categorification of a Lie triple system.  Finally we study skeletal and strict Lie triple $2$-systems. We show that skeletal   Lie triple $2$-systems can be classified the third cohomology group, and strict Lie triple $2$-systems are equivalent to crossed modules of Lie triple systems.
\end{abstract}

\keywords{Lie triple system, cohomology, Lie triple $2$-system, Nambu algebra, Leibniz algebra\\
\qquad 2020 Mathematics Subject Classification. 17A32, 17B10, 17B56, 17A42}

\maketitle

\tableofcontents

\allowdisplaybreaks

%\end{document}

\section{Introduction}

%As usual, we denote by $\mathbb{Z}$, $\mathbb{Z}_+$, $\mathbb{N}$ and
%$\mathbb{C}$ the sets of  all integers, nonnegative integers,
%positive integers and complex numbers, respectively. We assume that all the vector spaces are over $\mathbb{C}$.

\vspace{2mm}

Lie triple systems were originated from the research of symmetric spaces \cite{Cartan}. Jacobson firstly studied this system algebraically and named it Lie triple system \cite{J}. Lister constructed a structure theory of Lie triple systems in \cite{L}. The representation theory of Lie triple systems was given in \cite{HP}. Lie triple systems are  very important algebraic structures and have closed connection with many other algebraic structures, such as Nambu algebras \cite{DT},  Leibniz algebras, Jordan algebras \cite{BDE} and Lie algebras \cite{S}.  On the one hand, a Lie triple system is a special Nambu algebra. On the other hand, there is a Leibniz algebra structure on the space of fundamental objects.  Lie triple systems also play important roles in numerical
analysis of differential equations \cite{MQZ}.

The cohomology of an algebraic structure is very important since it can endow invariants. In particular, the second cohomology group can classify deformation and extension problems.
   Harris developed a cohomology theory of Lie triple systems using the universal enveloping algebra the {\rm Ext} functor in \cite{Harris}. On the other hand, Yamaguti constructed the representation and cohomology theory for Lie triple systems in \cite{Y}, without going out of a Lie triple system into an enveloping Lie algebra. Then Zhang explained this cohomology using the Loday-Pirashvili cohomology of the corresponding Leibniz algebra of fundamental objects in \cite{Zhang}, which makes the cohomology of Lie-triple systems well understood. Deformations of Lie triple systems were further studied in \cite{KT,YBB}.  

   In general, there is another approach to give a cohomology theory of an algebraic structure, namely using the controlling algebra. A controlling algebra of an algebraic structure is a graded Lie algebra (sometimes, $L_\infty$-algebra) whose Maurer-Cartan elements are the given algebraic structure. For example, the controlling algebra for Lie algebra structures on a vector space $V$ is given by the Nijenhuis-Richardson bracket  $[\cdot,\cdot]_\NR$ on the graded vector space $\oplus_{n=1}^{+\infty}\Hom(\wedge^{n+1}V,V)$ \cite{NR}, and the Chevalley-Eilenberg coboundary operator $\dM_\CE$ of a Lie algebra $(V,\mu)$ can be obtained by $\dM_\CE f=(-1)^{k-1}[\mu,f]_\NR$ for all $f\in\Hom(\wedge^{k}V,V)$.

   The first purpose of this paper is devoted to study the controlling algebra of Lie triple systems, and apply it to the existing cohomology theory. Observe that a Lie triple system is a special Nambu algebra, while the controlling algebra of Nambu algebras was already given in \cite{Rot}. This motivate us to solve this problem by figuring out a subalgebra of the graded Lie algebra given in \cite{Rot}. We further justify it by showing that the coboundary operator of a Lie triple system with coefficients in itself can be characterized by this controlling algebra structure.

   As soon as there is the controlling algebra of an algebraic structure, we can introduce the homotopy of the given algebraic structure via replace a vector space by a graded vector space. For example, on the graded vector space $\oplus_{n=1}^{+\infty}\Hom(\Sym( \huaV^\bullet),\huaV^\bullet)$ associated to a graded vector space $\huaV^\bullet$, there is also the Nijenhuis-Richardson bracket which makes it into a graded Lie algebra. Maurer-Cartan elements of this graded Lie algebra is exactly homotopy Lie algebra structures (also called $L_\infty$-algebras). Guided by this philosophy, we introduce the notion of homotopy Nambu algebras and homotopy Lie triple systems. We show that a homotopy Nambu algebra and a homotopy Lie triple system naturally gives rise to a $Leibniz_\infty$-algebra, which generalizes the classical result to the homotopy version. See \cite{ZLS} for the study of 3-Lie-infty-algebras.

   Then we focus on the 2-term case and enrich 2-term homotopy Lie triple systems to be a 2-category. On the other hand, it is well known that 2-term homotopy algebras are equivalent to the categorification of the algebraic structure. Thus it is natural to study the categorification of Lie triple systems. See \cite{baez:2algebras} for more details of categorification of Lie algebras. We introduce the notion of Lie triple 2-systems and show that the 2-category of 2-term homotopy Lie triple systems is equivalent to the 2-category of Lie triple 2-systems. We further classify skeletal Lie triple 2-systems using the third cohomology group of a Lie triple system, and show that strict Lie triple 2-systems are equivalent to crossed modules of Lie triple systems. Recently,   2-term $L_\infty$-triple algebras and the categorification of Lie-Yamaguti algebras was given in \cite{ZL} using a different approach.

%The aim of this paper is to provide a model for Lie triple systems that satisfy the fundamental identity up to all higher homotopies, and give the categorification of Lie triple systems. It is well-known that there is a Leibniz algebra structure on the space of fundamental objects associated to a Lie triple system. We define homotopy Lie triple systems in such a way that there is a Leibniz$_\infty$-algebra structure on the graded vector space of fundamental objects. Furthermore, we put Lie triple system structures on $2$-vector spaces and obtain Lie triple $2$-systems. Similar to the relation between Lie $2$-algebras and $2$-term $L_\infty$-algebras, we prove that the $2$-category of $2$-term homotopy Lie triple systems and the $2$-category of Lie triple $2$-systems are equivalent. Finally, we study skeletal Lie triple $2$-systems and strict Lie triple $2$-systems. In particular, we show that strict Lie triple $2$-systems are equivalent to crossed modules of Lie triple system.

The paper is organized as follows. In Section \ref{sec:con}, we give the controlling algebra of Lie triple systems. In Section \ref{sec:hom}, we introduce the notions of homotopy Nambu algebras and homotopy Lie triple systems. In Section \ref{sec:equ}, we introduce the notion of Lie triple 2-systems, and show that the 2-category of 2-term homotopy Lie triple systems is equivalent to the 2-category of Lie triple 2-systems. In Section \ref{sec:ske}, first we  classify skeletal Lie triple 2-systems using the third cohomology group of a Lie triple system. Then we introduce the notion of crossed modules of Lie triple systems, and show that strict Lie triple 2-systems are equivalent to crossed modules of Lie triple systems.

\vspace{2mm}
\noindent
{\bf Acknowledgements. } This research is supported by NSFC (12371029).

\section{The controlling algebra of Lie triple systems}\label{sec:con}

In this section, we give the controlling algebra of Lie triple systems, i.e. a graded Lie algebra whose Maurer-Cartan elements are Lie triple systems. The main tools we use is the controlling algebras of Leibniz algebras and Nambu algebras.

A {\bf Leibniz algebra} is a vector space $\g$ endowed with a linear map $[\cdot,\cdot]_{\g}:\g\otimes\g\lon\g$ satisfying
\begin{eqnarray}
[x,[y,z]_{\g}]_{\g}=[[x,y]_{\g},z]_{\g}+[y,[x,z]_{\g}]_{\g},\,\,\,\,\forall x,y,z\in\g.
\end{eqnarray}

Let $\g$ be a vector space. Consider the graded vector space $CL^*(\g,\g)=\oplus_{n\ge 0}CL^n(\g,\g)=\oplus_{n\ge 0}\Hom(\otimes^{n+1}\g,\g)$. It is known that $CL^*(\g,\g)$ equipped with the {\bf Balavoine bracket}
\begin{eqnarray}\label{leibniz-bracket}
[P,Q]_\B=P\bar{\circ}Q-(-1)^{pq}Q\bar{\circ}P,\,\,\,\,\forall P\in CL^{p}(\g,\g),Q\in CL^{q}(\g,\g),
\end{eqnarray}
is a graded Lie algebra (\cite{Balavoine-1}), where $P\bar{\circ}Q\in CL^{p+q}(\g,\g)$ is defined by
$
P\bar{\circ}Q=\sum_{k=1}^{p+1}P\circ_k Q,
$
and $\circ_k$ is defined by
\begin{eqnarray}
\label{commutator}&&(P\circ_kQ)(x_1,\cdots,x_{p+q+1})\\
\nonumber &=&\sum_{\sigma\in\mathbb S_{(k-1,q)}}(-1)^{(k-1)q}(-1)^{\sigma}P(x_{\sigma(1)},\cdots,x_{\sigma(k-1)},Q(x_{\sigma(k)},\cdots,x_{\sigma(k+q-1)},x_{k+q}),x_{k+q+1},\cdots,x_{p+q+1}).
\end{eqnarray}
In particular, for $\Omega\in CL^{1}(\g,\g)$, we have
\begin{eqnarray*}
([\Omega,\Omega]_\B)(x_1,x_2,x_3)=2\Big(\Omega(\Omega(x_1,x_2),x_3)-\Omega(x_1,\Omega(x_2,x_3))
+\Omega(x_2,\Omega(x_1,x_3))\Big).
\end{eqnarray*}
Thus, $\Omega$ defines a Leibniz algebra structure if and only if $[\Omega,\Omega]_\B=0$.

 \begin{defi}{\rm (\cite{DT})}
A {\bf   Nambu algebra} is a vector space $\g$ endowed with a linear map $[\cdot,\cdot,\cdot]_\g:\otimes^3\g\lon\g$ satisfying
\begin{eqnarray}
{}&&[x,y,[z,w,u]_\g]_\g=[[x,y,z]_\g,w,u]_\g+[z,[x,y,w]_\g,u]_\g+[z,w,[x,y,u]_\g]_\g.
\end{eqnarray}
\end{defi}

Let $C^p(\g,\g)=\Hom(\otimes^{2p+1}\g,\g)$ and $C^*(\g,\g)=\oplus_{p\ge 0}C^{p}(\g,\g)$. Define a graded linear map $\Phi:C^p(\g,\g)\lon CL^p(\otimes^2\frkg,\otimes^2\frkg)$ by
\begin{eqnarray}\label{3-Leibniz-to-Leibniz-classical}
(\Phi f)(\frkX_1,\cdots,\frkX_p,x\otimes y):=f(\frkX_1,\cdots,\frkX_p,x)\otimes y+x\otimes  f(\frkX_1,\cdots,\frkX_p,y),
\end{eqnarray}
for all $\frkX_i=x_i\otimes y_i\in \otimes^2\g, x,y\in\g$ and $f\in C^p(\g,\g),$ which is obviously  a graded  injective linear map.
It was proved in \cite{Rot}
that $\Img\Phi$ is a graded Lie subalgebra of $(CL^*(\otimes^2\frkg,\otimes^2\frkg),[\cdot,\cdot]_\B)$.
Consequently, one can transfer the graded Lie algebra structure on $\Img\Phi$ to a graded Lie algebra structure $\Courant{\cdot,\cdot}$ on $C^*(\g,\g)$ via
\begin{eqnarray*}
\Courant{P,Q}  =\Phi^{-1}\big([\Phi(P),\Phi(Q)]_\B\big),\,\,\forall P\in C^p(\g,\g),~Q\in C^q(\g,\g).
\end{eqnarray*}
More precisely, the graded vector space $C^*(\g,\g)=\oplus_{p\ge 0}C^{p}(\g,\g)$ equipped with
the graded bracket operation $$\Courant{\cdot,\cdot}: C^p(\g,\g)\times C^q(\g,\g)\longrightarrow C^{p+q}(\g,\g)$$ defined by
$\Courant{P,Q}=P\circ Q-(-1)^{pq}Q\circ P$, where $P\circ Q\in C^{p+q}(\g,\g)$ is defined by
$
P \circ Q=\sum_{k=1}^{p+1}(-1)^{(k-1)q}P\circ_k Q,
$
and for all $\frkX_i=x_i\otimes y_i\in \otimes^2\g, x\in\g$, the operation $\circ_k,~k=1,\cdots,p,$ is defined by
\begin{eqnarray}\label{Nambu-1}
&&(P\circ_k Q)(\frkX_1,\dotsc,\frkX_{p+q},x)\\
\nonumber &=&\sum_{\sigma\in \mathbb S_{(k-1,q)}}
(-1)^{\sigma}P\left(\frkX_{\sigma(1)},\dotsc,\frkX_{\sigma(k-1)},Q(\frkX_{\sigma(k)},\dotsc,\frkX_{\sigma(k+q-1)},x_{k+q})\otimes y_{k+q},\frkX_{k+q+1},\dotsc,\frkX_{p+q},x\right)\\
\nonumber &&+\sum_{\sigma\in \mathbb S_{(k-1,q)}}
(-1)^{\sigma}P\left(\frkX_{\sigma(1)},\dotsc,\frkX_{\sigma(k-1)},x_{k+q}\otimes Q(\frkX_{\sigma(k)},\dotsc,\frkX_{\sigma(k+q-1)},y_{k+q}) ,\frkX_{k+q+1},\dotsc,\frkX_{p+q},x\right),
\end{eqnarray}
and $\circ_{p+1}$ is defined by
\begin{eqnarray}\label{Nambu-2}
&&(P\circ_{p+1} Q)(\frkX_1,\dotsc,\frkX_{p+q},x)\\
\nonumber &=&\sum_{\sigma\in \mathbb S_{(p,q)}}
(-1)^{\sigma}P\left(\frkX_{\sigma(1)},\dotsc,\frkX_{\sigma(p)},Q(\frkX_{\sigma(p+1)},\dotsc,\frkX_{\sigma(p+q-1)}, \frkX_{\sigma(p+q)},x)\right),
\end{eqnarray}
is a graded Lie algebra.

\begin{thm}\label{Nambu-lie-algebra}\cite{Rot}
Maurer-Cartan elements of the graded Lie algebra $(C^*(\g,\g),\Courant{\cdot,\cdot})$ are the Nambu  algebra structures on the vector space $\g$.
\end{thm}

\begin{proof}
For $\pi\in C^1(\g,\g),~\frkX_1,\frkX_{2}\in\g\otimes\g,~x\in\g,$ we have
\begin{eqnarray*}
\Courant{\pi,\pi}(\frkX_1,\frkX_{2},x)&=&2(\pi\circ\pi)(\frkX_1,\frkX_{2},x)\\
                                      &=&2\Big(\pi\big(\pi(x_1,y_1,x_2),y_2,x\big)+\pi\big(x_2,\pi(x_1,y_1,y_2),x\big)\Big)\\
                                      &&+2\Big(-\pi\big(x_1,y_1,\pi(x_2,y_2,x)\big)+\pi\big(x_2,y_2,\pi(x_1,y_1,x)\big)\Big),
\end{eqnarray*}
which indicates that $\pi$ is a Nambu algebra structure on the vector space $\g$ if and only if $\Courant{\pi,\pi}=0$. The proof is finished.
\end{proof}

\begin{cor}\label{Lie-homo}
With the above notations,  $\Phi$ is   homomorphism from the graded Lie algebra  $(C^*(\g,\g),\Courant{\cdot,\cdot})$ to the graded Lie algebra $(CL^*(\otimes^2\frkg,\otimes^2\frkg),[\cdot,\cdot]_\B)$. Consequently, there is a Leibniz algebra structure on $\g\otimes \g$ associated to a Nambu algebra $(\g,[\cdot,\cdot,\cdot]_\g)$, which is defined by
\begin{eqnarray}\label{Nam}
[x\otimes y,z\otimes w]:=[x,y,z]_\g\otimes w+z\otimes [x,y,w]_\g,\,\,\,\forall x,y,z,w\in\g.
\end{eqnarray}
\end{cor}
 See \cite{DT,Gautheron} for more details for the relation between Leibniz algebras and Nambu algebras and their applications.

 \begin{defi}\cite{J}
A {\bf Lie triple system} is a vector space $\g$ endowed with a linear map $[\cdot,\cdot,\cdot]_\g:\g\otimes\g\otimes\g\lon\g$ satisfying
\begin{eqnarray}
{\label{lts1}}&&[x,x,y]_\g=0,\\
{\label{lts2}}&&[x,y,z]_\g+[y,z,x]_\g+[z,x,y]_\g=0,\\
{\label{lts3}}&&[x,y,[z,w,u]_\g]_\g=[[x,y,z]_\g,w,u]_\g+[z,[x,y,w]_\g,u]_\g+[z,w,[x,y,u]_\g]_\g.
\end{eqnarray}
\end{defi}

We note that Lie triple systems are specific classes of Nambu algebras. In order to characterise the Lie triple system structures on the vector space $\g$, we should choose a graded Lie subalgebra of $(C^*(\g,\g),\Courant{\cdot,\cdot})$. For this purpose, consider the graded subspace $\frkC_{\LTS}^*(\g,\g)=\oplus_{p\ge 0}\frkC_{\LTS}^p(\g,\g)$ of $C^*(\g,\g)$, where the space $\frkC_{\LTS}^p(\g,\g)$  is the vector space of linear maps that satisfies the following conditions:
\begin{eqnarray}
\label{condition-1}&&Q(\frkX_1,\cdots,\frkX_{p-1},x,x,y)=0,\\
\label{condition-2}&&Q(\frkX_1,\cdots,\frkX_{p-1},x,y,z)+Q(\frkX_1,\cdots,\frkX_{p-1},y,z,x)+Q(\frkX_1,\cdots,\frkX_{p-1},z,x,y)=0.
\end{eqnarray}

\begin{thm}
With above notations, $(\frkC_{\LTS}^*(\g,\g),\Courant{\cdot,\cdot})$ is a graded Lie subalgebra of the graded Lie algebra $(C^*(\g,\g),\Courant{\cdot,\cdot})$. Moreover, its Maurer-Cartan elements are the Lie triple system structures on the vector space $\g$.
\end{thm}

\begin{proof}
 By \eqref{Nambu-1}, for any $P\in \frkC_{\LTS}^p(\g,\g),~Q\in \frkC_{\LTS}^q(\g,\g)$, we deduce that  $P\circ_k Q\in \frkC_{\LTS}^{p+q}(\g,\g),~ k=1,\cdots,p-1$.
 Moreover, we have
 \begin{eqnarray*}
&&\Big((-1)^{(p-1)q}P\circ_p Q+(-1)^{pq}P\circ_{p+1} Q\Big)(\frkX_1,\cdots,\frkX_{p+q-1},x,x,y)\\
&\stackrel{\eqref{Nambu-2}}{=}&(-1)^{(p-1)q}\sum_{\sigma\in \mathbb S_{(p-1,q)}}
(-1)^{\sigma}P\left(\frkX_{\sigma(1)},\dotsc,\frkX_{\sigma(p-1)},Q(\frkX_{\sigma(p)},\dotsc,\frkX_{\sigma(p+q-1)},x),x,y\right)\\
 &&+(-1)^{(p-1)q}\sum_{\sigma\in \mathbb S_{(p-1,q)}}
(-1)^{\sigma}P\left(\frkX_{\sigma(1)},\dotsc,\frkX_{\sigma(p-1)},x,Q(\frkX_{\sigma(p)},\dotsc,\frkX_{\sigma(p+q-1)},x) ,y\right)\\
&&+(-1)^{pq}\sum_{\sigma\in \mathbb S_{(p,q-1)}}
(-1)^{\sigma}P\left(\frkX_{\sigma(1)},\dotsc,\frkX_{\sigma(p)},Q(\frkX_{\sigma(p+1)},\dotsc,\frkX_{\sigma(p+q-1)}, x,x,y)\right)\\
&&+(-1)^{pq+q}\sum_{\sigma\in \mathbb S_{(p-1,q)}}
(-1)^{\sigma}P\left(\frkX_{\sigma(1)},\dotsc,\frkX_{\sigma(p-1)},x,x,Q(\frkX_{\sigma(p)},\dotsc,\frkX_{\sigma(p+q-1)},y)\right)\\
&\stackrel{\eqref{condition-2}}{=}&0.
\end{eqnarray*}
On the other hand, we have
\begin{eqnarray*}
&&\Big((-1)^{(p-1)q}P\circ_p Q+(-1)^{pq}P\circ_{p+1} Q\Big)(\frkX_1,\cdots,\frkX_{p+q-1},x,y,z)\\
&&+\Big((-1)^{(p-1)q}P\circ_p Q+(-1)^{pq}P\circ_{p+1} Q\Big)(\frkX_1,\cdots,\frkX_{p+q-1},y,z,x)\\
&&+\Big((-1)^{(p-1)q}P\circ_p Q+(-1)^{pq}P\circ_{p+1} Q\Big)(\frkX_1,\cdots,\frkX_{p+q-1},z,x,y)\\
&\stackrel{\eqref{Nambu-2}}{=}&\underbrace{(-1)^{(p-1)q}\sum_{\sigma\in \mathbb S_{(p-1,q)}}
(-1)^{\sigma}P\left(\frkX_{\sigma(1)},\dotsc,\frkX_{\sigma(p-1)},Q(\frkX_{\sigma(p)},\dotsc,\frkX_{\sigma(p+q-1)},x),y,z\right)}\\
 &&+(-1)^{(p-1)q}\sum_{\sigma\in \mathbb S_{(p-1,q)}}
(-1)^{\sigma}P\left(\frkX_{\sigma(1)},\dotsc,\frkX_{\sigma(p-1)},x,Q(\frkX_{\sigma(p)},\dotsc,\frkX_{\sigma(p+q-1)},y) ,z\right)\\
&&+\underline{(-1)^{pq}\sum_{\sigma\in \mathbb S_{(p,q-1)}}
(-1)^{\sigma}P\left(\frkX_{\sigma(1)},\dotsc,\frkX_{\sigma(p)},Q(\frkX_{\sigma(p+1)},\dotsc,\frkX_{\sigma(p+q-1)}, x,y,z)\right)}\\
&&+(-1)^{pq+q}\sum_{\sigma\in \mathbb S_{(p-1,q)}}
(-1)^{\sigma}P\left(\frkX_{\sigma(1)},\dotsc,\frkX_{\sigma(p-1)},x,y,Q(\frkX_{\sigma(p)},\dotsc,\frkX_{\sigma(p+q-1)},z)\right)\\
&&+(-1)^{(p-1)q}\sum_{\sigma\in \mathbb S_{(p-1,q)}}
(-1)^{\sigma}P\left(\frkX_{\sigma(1)},\dotsc,\frkX_{\sigma(p-1)},Q(\frkX_{\sigma(p)},\dotsc,\frkX_{\sigma(p+q-1)},y),z,x\right)\\
 &&+(-1)^{(p-1)q}\sum_{\sigma\in \mathbb S_{(p-1,q)}}
(-1)^{\sigma}P\left(\frkX_{\sigma(1)},\dotsc,\frkX_{\sigma(p-1)},y,Q(\frkX_{\sigma(p)},\dotsc,\frkX_{\sigma(p+q-1)},z),x\right)\\
&&+\underline{(-1)^{pq}\sum_{\sigma\in \mathbb S_{(p,q-1)}}
(-1)^{\sigma}P\left(\frkX_{\sigma(1)},\dotsc,\frkX_{\sigma(p)},Q(\frkX_{\sigma(p+1)},\dotsc,\frkX_{\sigma(p+q-1)},y,z,x)\right)}\\
&&+\underbrace{(-1)^{pq+q}\sum_{\sigma\in \mathbb S_{(p-1,q)}}
(-1)^{\sigma}P\left(\frkX_{\sigma(1)},\dotsc,\frkX_{\sigma(p-1)},y,z,Q(\frkX_{\sigma(p)},\dotsc,\frkX_{\sigma(p+q-1)},x)\right)}\\
&&+(-1)^{(p-1)q}\sum_{\sigma\in \mathbb S_{(p-1,q)}}
(-1)^{\sigma}P\left(\frkX_{\sigma(1)},\dotsc,\frkX_{\sigma(p-1)},Q(\frkX_{\sigma(p)},\dotsc,\frkX_{\sigma(p+q-1)},z),x,y\right)\\
&&+\underbrace{(-1)^{(p-1)q}\sum_{\sigma\in \mathbb S_{(p-1,q)}}
(-1)^{\sigma}P\left(\frkX_{\sigma(1)},\dotsc,\frkX_{\sigma(p-1)},z,Q(\frkX_{\sigma(p)},\dotsc,\frkX_{\sigma(p+q-1)},x),y\right)}\\
&&+\underline{(-1)^{pq}\sum_{\sigma\in \mathbb S_{(p,q-1)}}
(-1)^{\sigma}P\left(\frkX_{\sigma(1)},\dotsc,\frkX_{\sigma(p)},Q(\frkX_{\sigma(p+1)},\dotsc,\frkX_{\sigma(p+q-1)},z,x,y)\right)}\\
&&+(-1)^{pq+q}\sum_{\sigma\in \mathbb S_{(p-1,q)}}
(-1)^{\sigma}P\left(\frkX_{\sigma(1)},\dotsc,\frkX_{\sigma(p-1)},z,x,Q(\frkX_{\sigma(p)},\dotsc,\frkX_{\sigma(p+q-1)},y)\right)\\
&\stackrel{\eqref{condition-1}}{=}&0.
\end{eqnarray*}
Thus, we deduce that $P\circ Q\in \frkC_{\LTS}^{p+q}(\g,\g)$. Therefore, $(\frkC_{\LTS}^*(\g,\g),\Courant{\cdot,\cdot})$ is a graded Lie subalgebra of $(C^*(\g,\g),\Courant{\cdot,\cdot})$.

For $\pi\in \frkC_{\LTS}^1(\g,\g),x,y,z\in\g,$ we have
\begin{eqnarray*}
\pi(x,x,y)=0,\quad \pi(x,y,z)+\pi(y,z,x)+\pi(z,x,y)=0.
\end{eqnarray*}
Moreover, by Theorem \ref{Nambu-lie-algebra}, we deduce that $\pi$ is a Lie triple system on the vector space $\g$ if and only if $\Courant{\pi,\pi}=0$. The proof is finished.
\end{proof}
At the end of this section, we recall the cohomology theory of Lie triple systems, and show that it can be recovered from the above controlling algebra.

\begin{defi}\cite{Y}
A {\bf representation} of a Lie triple system $(\frkg,[\cdot,\cdot,\cdot]_\g)$ on a vector space $V$ is a pair $(V;\rho)$, where $\rho:\otimes^2\frkg\longrightarrow\gl(V)$ is a linear map satisfying
\begin{eqnarray*}
\label{rep-1}&&[\rho(x_2,x_1)-\rho(x_1,x_2),\rho(y_1,y_2)]=\rho([x_1,x_2,y_1]_\g,y_2)+\rho(y_1,[x_1,x_2,y_2]_\g),\\
\label{rep-2}&&\rho(x_1,[y_1, y_2, y_3]_\g)=\rho(y_2, y_3)\rho(x_1,y_1) - \rho(y_1, y_3)\rho(x_1,y_2) + \big(\rho(y_2,y_1)-\rho(y_1,y_2)\big)\rho(x_1,y_3).
\end{eqnarray*}
\end{defi}

\begin{ex}
Let $(\g,[\cdot,\cdot,\cdot]_\g)$ be a Lie triple system. The linear map $\ad:\otimes^{2}\g\longrightarrow\gl(\g)$ defined by
$$
\ad_{x,y}z=[z,x,y]_\g,\quad \forall x,y,z\in\g
$$
is a representation  of the Lie triple system $(\frkg,[\cdot,\cdot,\cdot]_\g)$, which we call the {\bf adjoint representation}.
\end{ex}

An $n$-cochain on a Lie triple system $(\frkg,[\cdot,\cdot,\cdot]_\g)$ with the coefficient in a representation $(V;\rho)$ is a linear map
$f:\otimes^{2n-1}\g  \longrightarrow V$ satisfying
\begin{eqnarray}
\label{cohomo-1}&&f(\frkX_1,\cdots,\frkX_{n-2},x,x,y)=0,\\
\label{cohomo-2}&&f(\frkX_1,\cdots,\frkX_{n-2},x,y,z)+f(\frkX_1,\cdots,\frkX_{n-2},y,z,x)+f(\frkX_1,\cdots,\frkX_{n-2},z,x,y)=0,
\end{eqnarray}
for all $\frkX_i\in\g\otimes\g,~x,y,z\in\g.$

For all $n\ge 1,$ denote the space of $n$-cochains by $C^{n}_{\LTS}(\g;V).$ The coboundary operator $\delta:C^{n}_{\LTS}(\g;V)\longrightarrow C^{n+1}_{\LTS}(\g;V)$ is given by
\begin{equation}\label{eq:drho}
\begin{split}
&(\delta f)(\frkX_1,\dotsc ,\frkX_n,z)\\
={}& \sum_{1\leq j<k\leq n}(-1)^jf(\frkX_1,\dotsc ,\hat{\frkX}_j,\dotsc ,\frkX_{k-1},[\frkX_j,\frkX_k],\frkX_{k+1},\dotsc ,\frkX_{n},z)\\
&+ \sum_{j=1}^n(-1)^jf(\frkX_1,\dotsc ,\hat{\frkX}_j,\dotsc ,\frkX_{n},[\frkX_j,z]_\g)\\
&+ \sum_{j=1}^n(-1)^{j+1}\big(\rho(y_j,x_j)-\rho(x_j,y_j)\big)f(\frkX_1,\dotsc ,\hat{\frkX}_j,\dotsc ,\frkX_{n},z)\\
&+ (-1)^{n+1}\big(\rho(y_{n},z)f(\frkX_1,\dotsc ,\frkX_{n-1},x_{n} ) -\rho(x_{n},z)f(\frkX_1,\dotsc ,\frkX_{n-1},y_{n} ) \big),
\end{split}
\end{equation}
for all $\frkX_i=x_i\otimes y_i\in\otimes^2\g,~i=1,2\dotsc,n$ and $z\in\g$. $[\frkX_j,\frkX_k]$ is defined by \eqref{Nam}. It is proved in \cite{Y} that $\delta\circ\delta=0$.
\emptycomment{
A $p$-cochain on $\frkg$ with the coefficients in a representation $(V, \theta)$ is a linear map $\omega\in \frkC_{\LTS}^{p-1}(\g,V)$. And the coboundary operator $\delta:\frkC_{\LTS}^{p-1}(\g,V)\longrightarrow \frkC_{\LTS}^{p}(\g,V)$  is given by
  \begin{eqnarray*}
    &&(\delta \omega)(x_1,\cdots, x_{2p+1})\\
     &=& \theta(x_{2p},x_{2p+1})\omega(x_1,\cdots,x_{2p-1})-\theta(x_{2p-1},x_{2p+1})\omega(x_1,\cdots,x_{2p-2},x_{2p})\\
    &&+\sum_{k=1}^{p}(-1)^{k+p}D(x_{2k-1},x_{2k})\omega(x_1,\cdots,\widehat{ x_{2k-1}}\widehat{x_{2k}}, \cdots, x_{2p+1}) \\
     && +\sum_{k=1}^{p}\sum_{j=2k+1}^{2p+1}(-1)^{p+k+1}\omega(x_1,\cdots,\widehat{ x_{2k-1}}\widehat{x_{2k}}, \cdots,[x_{2k-1},x_{2k},x_j]_\g,\cdots, x_{2p+1}).
  \end{eqnarray*}
for all $x_i\in\frkg$.
}

We assume that $(\g,[\cdot,\cdot,\cdot]_\g)$ is a Lie triple system  and set $\pi(x,y,z)=[x,y,z]_\g$. Consider the coboundary operator
$\delta:C_{\LTS}^{n}(\g;\g)\longrightarrow C_{\LTS}^{n+1}(\g;\g)$ associated to the adjoint representation, then we have
\begin{thm}
For all $f\in C_{\LTS}^{n}(\g;\g)$, we have $$\delta f=(-1)^{n-1}\Courant{\pi, f},~\forall n=1,2,\dots.$$
\end{thm}
\begin{proof}
For all $\frkX_i=x_i\otimes y_i\in\otimes^2\g,~i=1,2\dotsc,n$ and $z\in\g$, we have
\begin{eqnarray*}
&&(-1)^{n-1}\Courant{\pi, f}(\frkX_1,\dotsc ,\frkX_n,z)\\
&=&(-1)^{n-1}(\pi\circ f-(-1)^{n-1}f\circ \pi)(\frkX_1,\dotsc ,\frkX_n,z)\\
&=&(-1)^{n-1}\Bigg(\sum_{\sigma\in \mathbb S_{(0,n-1)}}(-1)^{\sigma}\pi\left(f(\frkX_{\sigma(1)},\dotsc,\frkX_{\sigma(n-1)},x_n)\otimes y_n,z\right)\\
&+&\sum_{\sigma\in \mathbb S_{(0,n-1)}}(-1)^{\sigma}\pi\left(x_n\otimes f(\frkX_{\sigma(1)},\dotsc,\frkX_{\sigma(n-1)},y_n),z\right)\\
&+&(-1)^{n-1}\sum_{\sigma\in \mathbb S_{(1,n-1)}}(-1)^{\sigma}\pi\left(\frkX_{\sigma(1)}, f(\frkX_{\sigma(2)},\dotsc,\frkX_{\sigma(n)},z)\right)\\
&-&(-1)^{n-1}\bigg(\sum_{k=1}^{n-1}(-1)^{k-1}\Big(\sum_{\sigma\in \mathbb S_{(k-1,1)}}
(-1)^{\sigma}f\left(\frkX_{\sigma(1)},\dotsc,\frkX_{\sigma(k-1)},x_{k+1}\otimes \pi(\frkX_{\sigma(k)},y_{k+1}) ,\frkX_{k+2},\dotsc,\frkX_{n},z\right)\\
&+&\sum_{\sigma\in \mathbb S_{(k-1,1)}}
(-1)^{\sigma}f\left(\frkX_{\sigma(1)},\dotsc,\frkX_{\sigma(k-1)},\pi(\frkX_{\sigma(k)},x_{k+1})\otimes y_{k+1},\frkX_{k+2},\dotsc,\frkX_{n},z\right)\Big)\\
&+&(-1)^{n-1}\sum_{\sigma\in \mathbb S_{(n-1,1)}}(-1)^{\sigma}f\left(\frkX_{\sigma(1)},\dotsc,\frkX_{\sigma(n-1)},\pi(\frkX_{\sigma(n)},z)\right)\bigg)\Bigg)\\
&=&(-1)^{n-1}\big(\rho(y_{n},z)f(\frkX_1,\dotsc,\frkX_{n-1},x_{n})-\rho(x_{n},z)f(\frkX_1,\dotsc,\frkX_{n-1},y_{n})\big)\\
&+&\sum_{j=1}^n(-1)^{j+1}\big(\rho(y_j,x_j)-\rho(x_j,y_j)\big)f(\frkX_1,\dotsc,\hat{\frkX}_j,\dotsc,\frkX_{n},z)\\
&+&\sum_{1\leq j<k\leq n}(-1)^jf(\frkX_1,\dotsc,\hat{\frkX}_j,\dotsc,\frkX_{k-1},[\frkX_j,\frkX_k],\frkX_{k+1},\dotsc,\frkX_{n},z)\\
&+&\sum_{j=1}^n(-1)^jf(\frkX_1,\dotsc,\hat{\frkX}_j,\dotsc,\frkX_{n},[\frkX_j,z]_\g)\\
&=&(\delta f)(\frkX_1,\dotsc ,\frkX_n,z),
\end{eqnarray*}
which finishes the proof.
\end{proof}

\section{Homotopy Nambu algebras and homotopy Lie triples systems}\label{sec:hom}

Let $\huaV^\bullet$ be a graded vector space. Denote by $\Ten(\huaV^\bullet)=\oplus_{i=0}^{+\infty}(\otimes^i\huaV^\bullet)$  and $\bar{\Ten}(\huaV^\bullet)=\oplus_{i=1}^{+\infty}(\otimes^i\huaV^\bullet)$. For two graded vector spaces $\huaW_1^\bullet$ and $\huaW_2^\bullet$, denote by $\Hom^n(\huaW_1^\bullet,\huaW_2^\bullet)$ the space of degree $n$ linear maps from the graded vector space $\huaW_1^\bullet$ to the graded vector space $\huaW_2^\bullet$. Obviously, an element $f\in\Hom^n(\bar{\Ten}(\huaV^\bullet),\huaV^\bullet)$ is the sum of $f_i:{{\otimes^ i}\huaV^\bullet}\lon \huaV^\bullet$. We will write  $f=\sum_{i=1}^{+\infty} f_i$.
 Set $\CLV^n(\huaV^\bullet,\huaV^\bullet):=\Hom^n(\bar{\Ten}(\huaV^\bullet),\huaV^\bullet)$ and
$
\CLV^\bullet(\huaV^\bullet,\huaV^\bullet):=\oplus_{n\in\mathbb Z}\CLV^n(\huaV^\bullet,\huaV^\bullet).
$
As the graded version of the Balavoine bracket given in \cite{Balavoine-1}, the {\bf graded Balavoine bracket} $[\cdot,\cdot]_{\B}$ on the graded vector space $\CLV^\bullet(\huaV^\bullet,\huaV^\bullet)$ is given
by:
\begin{eqnarray}
\qquad [f,g]_{\B}:=f\bar{\circ} g-(-1)^{mn}g\bar{\circ} f,\,\,\,\,\forall f=\sum_{i=1}^{+\infty} f_i\in \CLV^m(\huaV^\bullet,\huaV^\bullet),~g=\sum_{j=1}^{+\infty}g_j\in \CLV^n(\huaV^\bullet,\huaV^\bullet),
\label{eq:gfgcirc-B}
\end{eqnarray}
where $f\bar{\circ} g\in \CLV^{m+n}(\huaV^\bullet,\huaV^\bullet)$ is defined by
 \begin{eqnarray}\label{graded-NR-circ-B}
f\bar{\circ} g&=&\Big(\sum_{i=1}^{+\infty}f_i\Big)\bar{\circ}\Big(\sum_{j=1}^{+\infty}g_j\Big):=\sum_{s=1}^{+\infty}\Big(\sum_{i+j=s+1}f_i\bar{\circ} g_j\Big),
\end{eqnarray}
while $f_i\bar{\circ} g_j\in \Hom({{\otimes ^s}\huaV^\bullet},\huaV^\bullet)$ is defined by
$f_i\bar{\circ} g_j=\sum_{k=1}^{i}f_i\bar{\circ}_k g_j
$
and $f_i\bar{\circ}_k g_j$ is defined by
\begin{equation}\label{commutator2}
\begin{split}
&(f_i\bar{\circ}_k g_j)(v_1,\cdots,v_{s})\\
={}& \sum_{\sigma\in\mathbb S_{(k-1,j-1)}}(-1)^{\beta_k}\varepsilon(\sigma)f_i(v_{\sigma(1)},\cdots,v_{\sigma(k-1)},g_j(v_{\sigma(k)},\cdots,v_{\sigma(k+j-2)},v_{k+j-1}),v_{k+j},\cdots,v_{s}),
\end{split}
\end{equation}
where $\beta_k=n(v_{\sigma(1)}+v_{\sigma(2)}+\cdots+v_{\sigma(k-1)})$.

Similar as the classical case, $(\CLV^\bullet(\huaV^\bullet,\huaV^\bullet),[\cdot,\cdot]_{\B})$ is a graded Lie algebra.

The notion of a Leibniz$_\infty$-algebra was given in \cite{livernet} by Livernet, which was further studied  in \cite{ammardefiLeibnizalgebra,Leibniz2al,STZ,Uchino-1}.
 \begin{defi}
A {\bf  Leibniz$_\infty$-algebra} is a $\mathbb Z$-graded vector space $\la^\bullet=\oplus_{k\in\mathbb Z}\la^k$ equipped with a collection $(k\ge 1)$ of linear maps $\oprn_k:\otimes^k\la^\bullet\lon\la^\bullet$ of degree $1$ such that  $\sum_{k=1}^{+\infty}\oprn_k$ is an Maurer-Cartan element of the graded Lie algebra $(\CLV^\bullet(\la^\bullet,\la^\bullet),[\cdot,\cdot]_{\B})$. More precisely, for any homogeneous elements $x_1,\cdots,x_n\in \la^\bullet$, the following equality holds:
\begin{eqnarray*}
\sum_{i=1}^{n}\sum_{k=1}^{n-i+1}\sum_{\sigma\in \mathbb S_{(k-1,i-1)} }(-1)^{\gamma_{k}}\varepsilon(\sigma)\oprn_{n-i+1}(x_{\sigma(1)},\cdots,x_{\sigma(k-1)},\oprn_i(x_{\sigma(k)},\cdots,x_{\sigma(k+i-2)},x_{k+i-1}),x_{k+i},\cdots,x_{n})=0,
\end{eqnarray*}
where $\gamma_{k}=x_{\sigma(1)}+\cdots+x_{\sigma(k-1)}$. We denote a Leibniz$_\infty$-algebra by $(\la^\bullet,\{\oprn_k\}_{k=1}^{+\infty})$.
\end{defi}

 Denote by $$\CV^n(\huaV^\bullet,\huaV^\bullet)=\oplus_{p=1}^{+\infty}\Hom^{n+p-1}(\otimes^{p-1}(\huaV^\bullet\otimes \huaV^\bullet)\otimes\huaV^\bullet,\huaV^\bullet),$$ which is the space of degree $n+p-1$ linear maps from the graded vector space $\oplus_{p=1}^{+\infty}(\otimes^{p-1}(\huaV^\bullet\otimes \huaV^\bullet)\otimes\huaV^\bullet)$ to the graded vector space $\huaV^\bullet$. Any $ f\in\oplus_{p=1}^{+\infty}\Hom^{n+p-1}(\otimes^{p-1}(\huaV^\bullet\otimes \huaV^\bullet)\otimes\huaV^\bullet,\huaV^\bullet)$ is the sum of $f_{2p-1}:\otimes^{p-1}(\huaV^\bullet\otimes \huaV^\bullet)\otimes \huaV^\bullet\lon \huaV^\bullet$. We will write  $f=\sum_{p=1}^{+\infty} f_{2p-1}$.
We define a graded linear map $\Psi:\CV^n(\huaV^\bullet,\huaV^\bullet)\lon \CLV^n(s(\huaV^\bullet\otimes \huaV^\bullet),s(\huaV^\bullet\otimes \huaV^\bullet))$ by
\begin{eqnarray}\label{3-Leibniz-to-Leibniz}
&&(\Psi f_{2p-1})(s\huaX_1,\cdots,s\huaX_{p-1},s(x\otimes y))\\
\nonumber &=&(-1)^{\frac{p(p-1)}{2}}s\Big(f_{2p-1}(\huaX_1,\cdots,\huaX_{p-1},x)\otimes y+(-1)^{x(\huaX_1+\cdots+\huaX_{p-1}+n+p-1)}x\otimes  f_{2p-1}(\huaX_1,\cdots,\huaX_{p-1},y)\Big),
\end{eqnarray}

for all $\huaX_i=x_i\otimes y_i\in \huaV^\bullet\otimes \huaV^\bullet, x,y\in\huaV^\bullet$.  Moreover, straightforward computation gives that $\Psi$ is an injective linear map.

\begin{thm}\label{Nambu-graded-Lie}
With the above notations,  $\Img\Psi$ is a graded Lie subalgebra of $(\CLV^\bullet(\huaV^\bullet\otimes \huaV^\bullet,\huaV^\bullet\otimes \huaV^\bullet),[\cdot,\cdot]_{\B})$.
\end{thm}

\begin{proof}
For any $f_i\in \Hom^m({{\otimes ^{i-1}}(\huaV^\bullet\otimes \huaV^\bullet)\otimes \huaV^\bullet},\huaV^\bullet),~g_j\in \Hom^n({{\otimes ^{j-1}}(\huaV^\bullet\otimes \huaV^\bullet)\otimes \huaV^\bullet},\huaV^\bullet)$ and $k=1,\cdots,i-1$,  we deduce that  $\Psi(f_i)\bar{\circ}_k\Psi(g_j)\in \Img\Psi$ from \eqref{commutator2}. Moreover, we have
{\footnotesize
\begin{eqnarray}
&&\Psi^{-1}\Big(\Psi(f_i)\bar{\circ}_k\Psi(g_j)\Big)(\huaX_1,\cdots,\huaX_{i+j-2},x)\\
\nonumber &=&\sum_{\sigma\in \mathbb S_{(k-1,j-1)}}
(-1)^{\beta_k}\varepsilon(\sigma)f_i\left(\huaX_{\sigma(1)},\dotsc,\huaX_{\sigma(k-1)},g_j(\huaX_{\sigma(k)},\dotsc,\huaX_{\sigma(k+j-2)},x_{k+j-1})\otimes y_{k+j-1},\huaX_{k+j},\dotsc,\huaX_{i+j-2},x\right)\\
\nonumber &&+\sum_{\sigma\in \mathbb S_{(k-1,j-1)}}
(-1)^{\beta_k+\gamma_k}\varepsilon(\sigma)f_i\left(\huaX_{\sigma(1)},\dotsc,\huaX_{\sigma(k-1)},x_{k+j-1}\otimes g_j(\huaX_{\sigma(k)},\dotsc,\huaX_{\sigma(k+j-2)},y_{k+j-1}) ,\huaX_{k+j},\dotsc,\huaX_{i+j-2},x\right),
\end{eqnarray}
}
where $\huaX_i=x_i\otimes y_i\in \huaV^\bullet \otimes \huaV^\bullet, x\in \huaV^\bullet$, $\beta_k=n(\huaX_{\sigma(1)}+\huaX_{\sigma(2)}+\cdots+\huaX_{\sigma(k-1)})$ and $\gamma_k=x_{k+j-1}(\huaX_{\sigma(k)}+\cdots+\huaX_{\sigma(k+j-2)}+n)$.
%For $k=i$, by straightforward computation, we gain that $\Psi(f_i)\bar{\circ}_{m+1}\Psi(g_j)-(-1)^{mn}\Psi(g_j)\bar{\circ}_{n+1}\Psi(f_i)\in\Img\Psi$
For $k=i$, we have
\begin{eqnarray*}
&&\Big(\Psi(f_i)\bar{\circ}_i\Psi(g_j)\Big)(s\huaX_1,\cdots,s\huaX_{i+j-2},s(x\otimes y))\\
&=&\sum_{\sigma\in \mathbb S_{(i-1,j-1)}}
(-1)^{\overline{\beta}_i}\varepsilon(\sigma)\Psi(f_i)\left(s\huaX_{\sigma(1)},\dotsc,s\huaX_{\sigma(i-1)},\Psi(g_j)(s\huaX_{\sigma(i)},\dotsc,s\huaX_{\sigma(i+j-2)},s(x\otimes y))\right)\\
&=&(-1)^{\frac{j(j-1)}{2}}\bigg(\sum_{\sigma\in \mathbb S_{(i-1,j-1)}}
(-1)^{\overline{\beta}_i}\varepsilon(\sigma)\Psi(f_i)\left(s\huaX_{\sigma(1)},\dotsc,s\huaX_{\sigma(i-1)},s(g_j(\huaX_{\sigma(i)},\dotsc,\huaX_{\sigma(i+j-2)},x)\otimes y)\right)\\
&&+\sum_{\sigma\in \mathbb S_{(i-1,j-1)}}
(-1)^{\overline{\beta}_i+\overline{\gamma}_i}\varepsilon(\sigma)\Psi(f_i)\left(s\huaX_{\sigma(1)},\dotsc,s\huaX_{\sigma(i-1)},s(x\otimes g_j(\huaX_{\sigma(i)},\dotsc,\huaX_{\sigma(i+j-2)}, y))\right)\bigg)\\
&=&(-1)^{\frac{i(i-1)+j(j-1)}{2}}s\bigg(\sum_{\sigma\in \mathbb S_{(i-1,j-1)}}
(-1)^{\overline{\beta}_i}\varepsilon(\sigma)f_i\left(\huaX_{\sigma(1)},\dotsc,\huaX_{\sigma(i-1)},g_j(\huaX_{\sigma(i)},\dotsc,\huaX_{\sigma(i+j-2)},x)\right)\otimes y\\
&&+\sum_{\sigma\in \mathbb S_{(i-1,j-1)}}
(-1)^{\overline{\beta}_i}(-1)^{(\huaX_{\sigma(i)}+\cdots+\huaX_{\sigma(i+j-2)}+x+n)(\huaX_{\sigma(1)}+\cdots+\huaX_{\sigma(i-1)}+m)}\varepsilon(\sigma)\\
&&g_j(\huaX_{\sigma(i)},\dotsc,\huaX_{\sigma(i+j-2)},x)\otimes f_i\left(\huaX_{\sigma(1)},\dotsc,\huaX_{\sigma(i-1)}, y\right)\\
&&+\sum_{\sigma\in \mathbb S_{(i-1,j-1)}}
(-1)^{\overline{\beta}_i+\overline{\gamma}_i}\varepsilon(\sigma)f_i\left(\huaX_{\sigma(1)},\dotsc,\huaX_{\sigma(i-1)},x\right)\otimes g_j(\huaX_{\sigma(i)},\dotsc,\huaX_{\sigma(i+j-2)}, y)\\
&&+\sum_{\sigma\in \mathbb S_{(i-1,j-1)}}
(-1)^{\overline{\beta}_i+\overline{\gamma}_i}(-1)^{x(\huaX_{\sigma(1)}+\cdots+\huaX_{\sigma(i-1)}+m)}\varepsilon(\sigma)\\
&&x\otimes f_i\left(\huaX_{\sigma(1)},\dotsc,\huaX_{\sigma(i-1)},g_j(\huaX_{\sigma(i)},\dotsc,\huaX_{\sigma(i+j-2)}, y)\right)\bigg)
\end{eqnarray*}
where $\overline{\beta}_i=n(\huaX_{\sigma(1)}+\huaX_{\sigma(2)}+\cdots+\huaX_{\sigma(i-1)})$ and $\overline{\gamma}_i=x(\huaX_{\sigma(i)}+\cdots+\huaX_{\sigma(i+j-2)}+n)$.
Similarly, we can gain
\begin{eqnarray*}
&&\Big(\Psi(g_j)\bar{\circ}_j\Psi(f_i)\Big)(s\huaX_1,\cdots,s\huaX_{i+j-2},s(x\otimes y))\\
&=&(-1)^{\frac{i(i-1)+j(j-1)}{2}}s\bigg(\sum_{\sigma\in \mathbb S_{(j-1,i-1)}}
(-1)^{\widetilde{\beta}_j}\varepsilon(\sigma)g_j\left(\huaX_{\sigma(1)},\dotsc,\huaX_{\sigma(j-1)},f_i(\huaX_{\sigma(j)},\dotsc,\huaX_{\sigma(i+j-2)},x)\right)\otimes y\\
&&+\sum_{\sigma\in \mathbb S_{(j-1,i-1)}}
(-1)^{\widetilde{\beta}_j}(-1)^{(\huaX_{\sigma(j)}+\cdots+\huaX_{\sigma(i+j-2)}+x+m)(\huaX_{\sigma(1)}+\cdots+\huaX_{\sigma(j-1)}+n)}\varepsilon(\sigma)\\
&&f_i(\huaX_{\sigma(j)},\dotsc,\huaX_{\sigma(i+j-2)},x)\otimes g_j\left(\huaX_{\sigma(1)},\dotsc,\huaX_{\sigma(j-1)},y\right)\\
&&+\sum_{\sigma\in \mathbb S_{(j-1,i-1)}}
(-1)^{\widetilde{\beta}_j+\widetilde{\gamma}_j}\varepsilon(\sigma)g_j\left(\huaX_{\sigma(1)},\dotsc,\huaX_{\sigma(j-1)},x\right)\otimes f_i(\huaX_{\sigma(j)},\dotsc,\huaX_{\sigma(i+j-2)}, y)\\
&&+\sum_{\sigma\in \mathbb S_{(j-1,i-1)}}
(-1)^{\widetilde{\beta}_j+\widetilde{\gamma}_j}(-1)^{x(\huaX_{\sigma(1)}+\cdots+\huaX_{\sigma(j-1)}+n)}\varepsilon(\sigma)\\
&&x\otimes g_j\left(\huaX_{\sigma(1)},\dotsc,\huaX_{\sigma(j-1)},f_i(\huaX_{\sigma(j)},\dotsc,\huaX_{\sigma(i+j-2)}, y)\right)\bigg)
\end{eqnarray*}
where $\widetilde{\beta}_j=m(\huaX_{\sigma(1)}+\huaX_{\sigma(2)}+\cdots+\huaX_{\sigma(j-1)})$ and $\widetilde{\gamma}_j=x(\huaX_{\sigma(j)}+\cdots+\huaX_{\sigma(i+j-2)}+m)$.
Thus, we obtain
\begin{eqnarray*}
&&\Big(\Psi(f_i)\bar{\circ}_i\Psi(g_j)-(-1)^{mn}\Psi(g_j)\bar{\circ}_j\Psi(f_i)\Big)(s\huaX_1,\cdots,s\huaX_{i+j-2},s(x\otimes y))\\
&=&(-1)^{\frac{i(i-1)+j(j-1)}{2}}s\bigg(\sum_{\sigma\in \mathbb S_{(i-1,j-1)}}
(-1)^{\overline{\beta}_i}\varepsilon(\sigma)f_i\left(\huaX_{\sigma(1)},\dotsc,\huaX_{\sigma(i-1)},g_j(\huaX_{\sigma(i)},\dotsc,\huaX_{\sigma(i+j-2)},x)\right)\otimes y\\
&&+\sum_{\sigma\in \mathbb S_{(i-1,j-1)}}
(-1)^{\overline{\beta}_i}(-1)^{x(\huaX_{\sigma(1)}+\cdots+\huaX_{\sigma(i+j-2)}+m+n)}\varepsilon(\sigma)\\
&&x\otimes f_i\left(\huaX_{\sigma(1)},\dotsc,\huaX_{\sigma(i-1)},g_j(\huaX_{\sigma(i)},\dotsc,\huaX_{\sigma(i+j-2)}, y)\right)\\
&&-\sum_{\sigma\in \mathbb S_{(j-1,i-1)}}
(-1)^{mn}(-1)^{\widetilde{\beta}_j}\varepsilon(\sigma)g_j\left(\huaX_{\sigma(1)},\dotsc,\huaX_{\sigma(j-1)},f_i(\huaX_{\sigma(j)},\dotsc,\huaX_{\sigma(i+j-2)},x)\right)\otimes y\\
&&-\sum_{\sigma\in \mathbb S_{(j-1,i-1)}}
(-1)^{mn}(-1)^{\widetilde{\beta}_j}(-1)^{x(\huaX_{\sigma(1)}+\cdots+\huaX_{\sigma(i+j-2)}+m+n)}\varepsilon(\sigma)\\
&&x\otimes g_j\left(\huaX_{\sigma(1)},\dotsc,\huaX_{\sigma(j-1)},f_i(\huaX_{\sigma(j)},\dotsc,\huaX_{\sigma(i+j-2)}, y)\right)\bigg)
\end{eqnarray*}
Therefore,  $\Psi(f_i)\bar{\circ}_i\Psi(g_j)-(-1)^{mn}\Psi(g_j)\bar{\circ}_j\Psi(f_i)\in\Img\Psi$. More precisely, we have
\begin{eqnarray}
\label{LTS-circ-3}&&\Psi^{-1}\Big(\Psi(f_i)\bar{\circ}_i\Psi(g_j)-(-1)^{mn}\Psi(g_j)\bar{\circ}_j\Psi(f_i)\Big)(\huaX_1,\cdots,\huaX_{i+j-2},x)\\
\nonumber &=&\sum_{\sigma\in \mathbb S_{(i-1,j-1)}}
(-1)^{\beta_i}\varepsilon(\sigma)f_i\left(\huaX_{\sigma(1)},\dotsc,\huaX_{\sigma(i-1)},g_j(\huaX_{\sigma(i)},\dotsc,\huaX_{\sigma(i+j-2)},x)\right)\\
\nonumber &&-\sum_{\sigma\in \mathbb S_{(j-1,i-1)}}
(-1)^{mn}(-1)^{\beta_j}\varepsilon(\sigma)g_j\left(\huaX_{\sigma(1)},\dotsc,\huaX_{\sigma(j-1)},f_i(\huaX_{\sigma(j)},\dotsc,\huaX_{\sigma(i+j-2)},x)\right).
\end{eqnarray}
Moreover, by  \eqref{eq:gfgcirc-B}, we deduce  that $[\Psi(f_i),\Psi(g_j)]_\B\in \Img\Psi$. The proof is finished.
\end{proof}

Consider the graded vector space $\CV^\bullet(\huaV^\bullet,\huaV^\bullet)=\oplus_{n\in\mathbb Z}\CV^n(\huaV^\bullet,\huaV^\bullet)$.
We define a graded bracket operation $$\Courant{\cdot,\cdot}: \CV^m(\huaV^\bullet,\huaV^\bullet)\times \CV^n(\huaV^\bullet,\huaV^\bullet)\longrightarrow \CV^{m+n}(\huaV^\bullet,\huaV^\bullet)$$ by
\begin{eqnarray*}
\Courant{f,g}=\Psi^{-1}\big([\Psi(f),\Psi(g)]_\B\big),\,\,\forall f\in \CV^m(\huaV^\bullet,\huaV^\bullet),~g\in \CV^n(\huaV^\bullet,\huaV^\bullet).
\end{eqnarray*}
More precisely,
$\Courant{f,g}=f\circ g-(-1)^{\deg(f)\deg(g)}g\circ f$, where $f\circ g\in \CV^{m+n}(\huaV^\bullet,\huaV^\bullet)$ is defined by
\begin{eqnarray*}
f\circ g&=&\sum_{s=0}^{+\infty}\Big(\sum_{i+j=s+2}f_i\circ g_j\Big),
\end{eqnarray*}
while $f_i\circ g_j\in \Hom({{\otimes ^s}(\huaV^\bullet\otimes \huaV^\bullet)\otimes \huaV^\bullet},\huaV^\bullet)$ is defined by
$f_i\circ g_j=\sum_{k=1}^{i}f_i\circ_k g_j$
and $\circ_k,~k=1,\cdots,i-1,$ is defined by
{\footnotesize
\begin{eqnarray*}
&&(f_i\circ_k g_j)(\huaX_1,\dotsc,\huaX_{s},x)\\
&=&\sum_{\sigma\in \mathbb S_{(k-1,j-1)}}
(-1)^{\beta_k}\varepsilon(\sigma)f_i\left(\huaX_{\sigma(1)},\cdots,\huaX_{\sigma(k-1)},g_j(\huaX_{\sigma(k)},\cdots,\huaX_{\sigma(k+j-2)},x_{k+j-1})\otimes y_{k+j-1},\huaX_{k+j},\cdots,\huaX_{s},x\right)\\
&&+\sum_{\sigma\in \mathbb S_{(k-1,j-1)}}
(-1)^{\beta_k+\gamma_k}\varepsilon(\sigma)f_i\left(\huaX_{\sigma(1)},\cdots,\huaX_{\sigma(k-1)},x_{k+j-1}\otimes g_j(\huaX_{\sigma(k)},\cdots,\huaX_{\sigma(k+j-2)},y_{k+j-1}),\huaX_{k+j},\cdots,\huaX_{s},x\right),
\end{eqnarray*}
}
and $\circ_{i}$ is defined by
\begin{eqnarray*}
&&(f_i\circ_{i} g_j)(\huaX_1,\dotsc,\huaX_{s},x)\\
&=&\sum_{\sigma\in \mathbb S_{(i-1,j-1)}}
(-1)^{\beta_i}\varepsilon(\sigma)f_i\left(\huaX_{\sigma(1)},\dotsc,\huaX_{\sigma(i-1)},g_j(\huaX_{\sigma(i)},\dotsc, \huaX_{\sigma(s)},x)\right),
\end{eqnarray*}
where $\huaX_i=x_i\otimes y_i\in \huaV^\bullet \otimes \huaV^\bullet, x\in \huaV^\bullet$, $\beta_k=(n+j-1)(\huaX_{\sigma(1)}+\huaX_{\sigma(2)}+\cdots+\huaX_{\sigma(k-1)})$ and $\gamma_k=x_{k+j-1}(\huaX_{\sigma(k)}+\cdots+\huaX_{\sigma(k+j-2)}+n+j-1)$.

\begin{thm}
With the above notations, $(\CV^\bullet(\huaV^\bullet,\huaV^\bullet),\Courant{\cdot,\cdot})$ is a graded Lie algebra.
\end{thm}
\begin{proof}
  It follows from Theorem \ref{Nambu-graded-Lie} directly.
\end{proof}

In the classical case, Nambu algebras are Maurer-Cartan elements of the graded Lie algebra given in Theorem \ref{Nambu-lie-algebra}.
Using this point of view, Maurer-Cartan elements of the graded Lie algebra $(\CV^\bullet(\huaV^\bullet,\huaV^\bullet),\Courant{\cdot,\cdot})$ obtained above should be homotopy Nambu  algebra structures on the graded vector space $\huaV^\bullet$. Thus we propose the following definition.

\begin{defi}
A {\bf homotopy Nambu algebra}  is a $\mathbb Z$-graded vector space $\huaV^\bullet=\oplus_{k\in\mathbb Z}\huaV^k$ equipped with a collection $(k\ge 1)$ of linear maps $\Oprn_{2k-1}:\otimes^{2k-1}\huaV^\bullet\lon\huaV^\bullet$ of degree $k$ such that $\sum_{k=1}^{+\infty}\Oprn_{2k-1}$ is an Maurer-Cartan element\footnote{$\Oprn_{2k-1}$ is of degree $k$ implies that $\Oprn_{2k-1}\in \CV^1(\huaV^\bullet,\huaV^\bullet).$} of the graded Lie algebra $(\CV^\bullet(\huaV^\bullet,\huaV^\bullet),\Courant{\cdot,\cdot})$.
More precisely, for any homogeneous elements $\huaX_1,\cdots,\huaX_{n-1}\in \huaV^\bullet\otimes \huaV^\bullet, x\in \huaV^\bullet$, the following equality holds:
\begin{eqnarray}
\label{eq:homNA}&&\sum_{i=1}^{n}\sum_{k=1}^{n-i}\sum_{\sigma\in \mathbb S_{(k-1,i-1)} }(-1)^{\gamma_{k}}\varepsilon(\sigma)\\
\nonumber&&\Oprn_{2n-2i+1}(\huaX_{\sigma(1)},\cdots,\huaX_{\sigma(k-1)},\oprn_i(\huaX_{\sigma(k)},\cdots,\huaX_{\sigma(k+i-2)},\huaX_{k+i-1}),\huaX_{k+i},\cdots,\huaX_{n-1},x)\\
\nonumber&&+\sum_{i=1}^{n}\sum_{\sigma\in \mathbb S_{(n-i,i-1)} }(-1)^{\gamma_{n-i+1}}\varepsilon(\sigma)\Oprn_{2n-2i+1}(\huaX_{\sigma(1)},\cdots,\huaX_{\sigma(n-i)},\Oprn_{2i-1}(\huaX_{\sigma(n-i+1)},\cdots,\huaX_{\sigma(n-1)},x))=0,
\end{eqnarray}
where $\gamma_{k}=i(\huaX_{\sigma(1)}+\cdots+\huaX_{\sigma(k-1)}),~k=1,2,\cdots,n-i+1,$ and $\oprn_i$ is defined by
\begin{eqnarray}\label{homotopy-nabum-to-homotopy-leibniz}
&&\oprn_i(\huaX_1,\cdots,\huaX_{i})\\
\nonumber &=&\Oprn_{2i-1}(\huaX_1,\cdots,\huaX_{i-1},x_i)\otimes y_i+(-1)^{x_i(\huaX_1+\cdots+\huaX_{i-1}+i)}x_i\otimes \Oprn_{2i-1}(\huaX_1,\cdots,\huaX_{i-1},y_i).
\end{eqnarray}
\end{defi}
 We denote a homotopy Nambu algebra by $(\huaV^\bullet,\{\Oprn_{2k-1}\}_{k=1}^{+\infty})$.

\begin{cor}
With the above notations, $\Psi$ is a  homomorphism from the  graded Lie algebra $(\CV^\bullet(\huaV^\bullet,\huaV^\bullet),\Courant{\cdot,\cdot})$ to $(\CLV^\bullet(s(\huaV^\bullet\otimes \huaV^\bullet),s(\huaV^\bullet\otimes \huaV^\bullet)),[\cdot,\cdot]_{\B})$.
Consequently, a  homotopy Nambu algebra $(\huaV^\bullet,\{\Oprn_{2k-1}\}_{k=1}^{+\infty})$ naturally gives rise to a Leibniz$_\infty$-algebra $(s(\huaV^\bullet\otimes \huaV^\bullet),\{\oprn_i\}_{i=1}^{+\infty})$, where $\oprn_i$ is given by \eqref{homotopy-nabum-to-homotopy-leibniz}.
\end{cor}

A Lie triple system is just a Nambu algebra satisfying further compatibility conditions. From this perspective, we introduce the notion of a homotopy Lie triple system as follows.

\begin{defi}\label{Homotopy-lts1}
A {\bf homotopy Lie triple system} is a $\mathbb Z$-graded vector space $\huaT^\bullet=\oplus_{k\in\mathbb Z}\huaT^k$ equipped with a collection $(k\ge 1)$ of linear maps $\Oprn_{2k-1}:\otimes^{2k-1}\huaT^\bullet\lon\huaT^\bullet$ of degree $k$ such that for any homogeneous elements $\huaX_1,\cdots,\huaX_{n-1}\in \huaT^\bullet\otimes \huaT^\bullet, x\in \huaT^\bullet$, \eqref{eq:homNA} and the following equalities hold:
\begin{eqnarray*}
\Oprn_{2k-1}(\huaX_1,\cdots,\huaX_{k-2},x,x,y)=0,\\
\Oprn_{2k-1}(\huaX_1,\cdots,\huaX_{k-2},x,y,z)+(-1)^{x(y+z)}\Oprn_{2k-1}(\huaX_1,\cdots,\huaX_{k-2},y,z,x)\\
+(-1)^{(x+y)z}\Oprn_{2k-1}(\huaX_1,\cdots,\huaX_{k-2},z,x,y)=0.
\end{eqnarray*}

\end{defi}
 We denote a homotopy Lie triple system by $(\huaT^\bullet,\{\Oprn_{2k-1}\}_{k=1}^{+\infty})$.

\section{2-term homotopy Lie triple systems and Lie triple $2$-systems}\label{sec:equ}

In this section, we focus on 2-term homotopy Lie triple systems and introduce the notion of Lie triple $2$-systems, which are the categorification of Lie triple systems. We show that 2-term homotopy Lie triple systems and Lie triple $2$-systems are equivalent.

\subsection{2-term homotopy Lie triple systems}
For all $i\ge 1$, let $l_{2i-1}:\big(\otimes^{i-1}(\huaV^\bullet\otimes\huaV^\bullet)\big)\otimes \huaV^\bullet\lon\huaV^\bullet$ be a graded linear map of degree $2-i$. Define $D(l_{2i-1}):\big(\otimes^{i-1}(s^{-1}\huaV^\bullet\otimes s^{-1}\huaV^\bullet)\big)\otimes s^{-1}\huaV^\bullet\lon s^{-1}\huaV^\bullet$ by
\begin{eqnarray*}
D(l_{2i-1})=(-1)^{(2i-1)(i-1)}s^{-1}\circ l_{2i-1}\circ s^{\otimes 2i-1},
\end{eqnarray*}
which is a graded linear map of degree $i$, so that $D(l_{2i-1}) \in \CV^1(s^{-1}\huaV^\bullet,s^{-1}\huaV^\bullet)$. This defines an isomorphism called the d\'ecalage isomorphism. Using this isomorphism, we introduce another version of homotopy Lie triple systems, which is more closely related to the cohomology theory.

\begin{defi}\label{Homotopy-lts2}
A {\bf homotopy Lie triple system} is a $\mathbb Z$-graded vector space $\huaT^\bullet=\oplus_{k\in\mathbb Z}\huaT^k$ equipped with a collection $(k\ge 1)$ of linear maps $l_{2k-1}:\otimes^{2k-1}\huaT^\bullet\lon\huaT^\bullet$ of degree $2-k$ such that for any homogeneous elements $\huaX_1,\cdots,\huaX_{n-1}\in \huaT^\bullet\otimes \huaT^\bullet, x\in \huaT^\bullet$, the following equality holds:
\begin{eqnarray*}
l_{2k-1}(\huaX_1,\cdots,\huaX_{k-2},x,x,y)=0,\\
l_{2k-1}(\huaX_1,\cdots,\huaX_{k-2},x,y,z)+(-1)^{x(y+z)}l_{2k-1}(\huaX_1,\cdots,\huaX_{k-2},y,z,x)\\
+(-1)^{(x+y)z}l_{2k-1}(\huaX_1,\cdots,\huaX_{k-2},z,x,y)=0,
\end{eqnarray*}
\begin{eqnarray}\label{homotopy-lts2}
\nonumber&&\sum_{i=1}^{n}\sum_{k=1}^{n-i}\sum_{\sigma\in \mathbb S_{(k-1,i-1)} }(-1)^{\gamma_{k}}\sgn(\sigma)\varepsilon(\sigma)\\
\nonumber&&l_{2n-2i+1}(\huaX_{\sigma(1)},\cdots,\huaX_{\sigma(k-1)},\widetilde{l}_i(\huaX_{\sigma(k)},\cdots,\huaX_{\sigma(k+i-2)},\huaX_{k+i-1}),\huaX_{k+i},\cdots,\huaX_{n-1},x)\\
\nonumber&&+\sum_{i=1}^{n}\sum_{\sigma\in \mathbb S_{(n-i,i-1)} }(-1)^{\gamma_{n-i+1}}\sgn(\sigma)\varepsilon(\sigma)\\
\nonumber&&l_{2n-2i+1}(\huaX_{\sigma(1)},\cdots,\huaX_{\sigma(n-i)},l_{2i-1}(\huaX_{\sigma(n-i+1)},\cdots,\huaX_{\sigma(n-1)},x))=0,
\end{eqnarray}
%$$l_{2k-1}(\huaX_1,\cdots,\huaX_{k-2},x,x,y)=0,$$
%$$l_{2k-1}(\huaX_1,\cdots,\huaX_{k-2},x,y,z)+(-1)^{x(y+z)}l_{2k-1}(\huaX_1,\cdots,\huaX_{k-2},y,z,x)+(-1)^{(x+y)z}l_{2k-1}(\huaX_1,\cdots,\huaX_{k-2},z,x,y)=0,$$
where $\gamma_{k}=(2-i)(\huaX_{\sigma(1)}+\cdots+\huaX_{\sigma(k-1)}),~k=1,2,\cdots,n-i+1,$ and $\widetilde{l}_i$ is defined by
\begin{eqnarray}\label{homotopy-nabum-to-homotopy-leibniz2}
&&\widetilde{l}_i(\huaX_1,\cdots,\huaX_{i})\\
\nonumber &=&l_{2i-1}(\huaX_1,\cdots,\huaX_{i-1},x_i)\otimes y_i+(-1)^{x_i(\huaX_1+\cdots+\huaX_{i-1}+2-i)}x_i\otimes l_{2i-1}(\huaX_1,\cdots,\huaX_{i-1},y_i).
\end{eqnarray}

\end{defi}
 We denote a homotopy Lie triple system by $(\huaT^\bullet,\{l_{2k-1}\}_{k=1}^{+\infty})$.

\begin{lem}
Definition \ref{Homotopy-lts1} and Definition \ref{Homotopy-lts2} are equivalent.
\end{lem}

\begin{proof}
Let $(\huaT^\bullet,\{l_{2k-1}\}_{k=1}^{+\infty})$ be a homotopy Lie triple system as in Definition \ref{Homotopy-lts2}. We define linear maps
$\Oprn_{2i-1}=D(l_{2i-1})=(-1)^{(2i-1)(i-1)}s^{-1}\circ l_{2i-1}\circ s^{\otimes 2i-1}$.
It is easy to check that $(s^{-1}\huaT^\bullet,\{\Oprn_{2k-1}\}_{k=1}^{+\infty})$ is a homotopy Lie triple system as in Definition \ref{Homotopy-lts1}.
\end{proof}

\begin{defi}\label{2termHomLTS}
A $2$-term homotopy Lie triple system $\huaT=(\huaT_0,\huaT_{-1},\dM,[\cdot,\cdot,\cdot],J)$, consists of the following data:
\begin{itemize}
\item[$\bullet$] a complex of vector spaces $\dM: \huaT_{-1}\rightarrow \huaT_0$
\item[$\bullet$] a trilinear map $[\cdot,\cdot,\cdot]: \huaT_i\times \huaT_j\times \huaT_k\rightarrow \huaT_{i+j+k}$, where $-1\leq i+j+k\leq 0$
\item[$\bullet$] a multilinear map $J: \huaT_0\times \huaT_0\times \huaT_0\times \huaT_0\times \huaT_0\rightarrow \huaT_{-1}$
\end{itemize}
such that for all $x,y,z,x_i\in \huaT_0$ and $f,g,h\in \huaT_{-1}$, the following equalities are satisfied:
\begin{itemize}
\item[$\rm(a)$] $\dM[x,y,f]=[x,y,\dM f]; \quad \dM[x,f,y]=[x,\dM f,y],$
\item[$\rm(b)$] $[\dM f,g,x]=[f,\dM g,x]; \quad [\dM f,x,g]=[f,x,\dM g],$
\item[$\rm(c)$] $[x,y,z]=-[y,x,z]; \quad [x,y,f]=-[y,x,f]; \quad [x,f,y]=-[f,x,y],$
\item[$\rm(d)$] $[x,y,z]+[y,z,x]+[z,x,y]=0,$
\item[$\rm(e)$] $[x,y,f]+[y,f,x]+[f,x,y]=0,$
\item[$\rm(f)$] $J(x_1,x_2,x,x,y)=0,$
\item[$\rm(g)$] $J(x_1,x_2,x,y,z)+J(x_1,x_2,y,z,x)+J(x_1,x_2,z,x,y)=0,$
\item[$\rm(h)$] $\dM J(x_1,x_2,x_3,x_4,x_5)\\=-[x_1,x_2,[x_3,x_4,x_5]]+[x_3,[x_1,x_2,x_4],x_5]+[[x_1,x_2,x_3],x_4,x_5]+[x_3,x_4,[x_1,x_2,x_5]],$
\item[$\rm(i)$] $J(\dM f,x_2,x_3,x_4,x_5)\\=-[f,x_2,[x_3,x_4,x_5]]+[x_3,[f,x_2,x_4],x_5]+[[f,x_2,x_3],x_4,x_5]+[x_3,x_4,[f,x_2,x_5]],$
\item[$\rm(j)$] $J(x_1,x_2,\dM f,x_4,x_5)\\=-[x_1,x_2,[f,x_4,x_5]]+[f,[x_1,x_2,x_4],x_5]+[[x_1,x_2,f],x_4,x_5]+[f,x_4,[x_1,x_2,x_5]],$
\item[$\rm(k)$] $J(x_1,x_2,x_3,x_4,\dM f)\\=-[x_1,x_2,[x_3,x_4,f]]+[x_3,[x_1,x_2,x_4],f]+[[x_1,x_2,x_3],x_4,f]+[x_3,x_4,[x_1,x_2,f]],$
\item[$\rm(l)$] $[J(x_1,x_2,x_3,x_4,x_5),x_6,x_7]+[x_5,J(x_1,x_2,x_3,x_4,x_6),x_7]+[x_1,x_2,J(x_3,x_4,x_5,x_6,x_7)]\\
+[x_5,x_6,J(x_1,x_2,x_3,x_4,x_7)]+J(x_1,x_2,[x_3,x_4,x_5],x_6,x_7)+J(x_1,x_2,x_5,[x_3,x_4,x_6],x_7)\\
+J(x_1,x_2,x_5,x_6,[x_3,x_4,x_7])=[x_3,x_4,J(x_1,x_2,x_5,x_6,x_7)]+J([x_1,x_2,x_3],x_4,x_5,x_6,x_7)\\
+J(x_3,[x_1,x_2,x_4],x_5,x_6,x_7)+J(x_3,x_4,[x_1,x_2,x_5],x_6,x_7)+J(x_3,x_4,x_5,[x_1,x_2,x_6],x_7)\\
+J(x_1,x_2,x_3,x_4,[x_5,x_6,x_7])+J(x_3,x_4,x_5,x_6,[x_1,x_2,x_7]).$
\end{itemize}
\end{defi}
Equations (a) and (b) tells us how the differential $\dM$ and the bracket $[\cdot,\cdot,\cdot]$ interact. Equations (h), (i), (j) and (k) tell us that the fundamental identity no longer holds on the nose, but controlled by $J$, Equation (l) gives the coherence law that $J$ should satisfy.

\begin{defi}\label{defi:Lwuqiong hom}
Let $\huaT=(\huaT_0,\huaT_{-1},\dM,[\cdot,\cdot,\cdot],J)$ and $\huaT'=(\huaT_0',\huaT_{-1}',\dM',[\cdot,\cdot,\cdot]',J')$ be $2$-term homotopy Lie triple systems. A {\bf  homomorphism} $\phi:\huaT \longrightarrow \huaT'$ consists of:
\begin{itemize}
\item[$\bullet$] a chain map $\phi:\huaT \longrightarrow \huaT'$, which consists of linear maps $\phi_0:\huaT_0 \longrightarrow \huaT_0'$ and $\phi_1:\huaT_{-1} \longrightarrow \huaT_{-1}'$ preserving the differential;
\item[$\bullet$] a trilinear map $\phi_2:\huaT_0 \times \huaT_0 \times \huaT_0 \longrightarrow \huaT_{-1}'$,
\end{itemize}
such that for all $x_i\in \huaT_0$ and $h\in \huaT_{-1}$, we have
\begin{eqnarray}
 \label{eq:homo1} \phi_2(x_1,x_1,x_2)&=&0,\\
 \label{eq:homo2} \phi_2(x_1,x_2,x_3)&=&-\phi_2(x_2,x_3,x_1)-\phi_2(x_3,x_1,x_2),\\
 \label{eq:homo3}\dM' (\phi_2(x_1,x_2,x_3))&=&\phi_0([x_1,x_2,x_3])-[\phi_0(x_1),\phi_0(x_2),\phi_0(x_3)]',\\
 \label{eq:homo4} \phi_2(x_1,x_2,\dM h)&=&\phi_1([x_1,x_2,h])-[\phi_0(x_1),\phi_0(x_2),\phi_1(h)]',\\
 \label{eq:homo5} \phi_2(x_1,\dM h,x_2)&=&\phi_1([x_1,h,x_2])-[\phi_0(x_1),\phi_1(h),\phi_0(x_2)]',
\end{eqnarray}
and
\begin{eqnarray}
    \label{eq:homo6}&&J'(\phi_0(x_1),\phi_0(x_2),\phi_0(x_3),\phi_0(x_4),\phi_0(x_5))+[\phi_2(x_1,x_2,x_3),\phi_0(x_4),\phi_0(x_5)]'\\
    \nonumber &+&[\phi_0(x_3),\phi_2(x_1,x_2,x_4),\phi_0(x_5)]'+[\phi_0(x_3),\phi_0(x_4),\phi_2(x_1,x_2,x_5)]'\\
    \nonumber &+&\phi_2([x_1,x_2,x_3],x_4,x_5)+\phi_2(x_3,[x_1,x_2,x_4],x_5)+\phi_2(x_3,x_4,[x_1,x_2,x_5])\\
    \nonumber &=&[\phi_0(x_1),\phi_0(x_2),\phi_2(x_3,x_4,x_5)]'+\phi_2(x_1,x_2,[x_3,x_4,x_5])+\phi_1(J(x_1,x_2,x_3,x_4,x_5)).
\end{eqnarray}
\end{defi}

Let $\varphi:\huaT \longrightarrow \huaT'$ and $\psi:\huaT' \longrightarrow \huaT''$ be  homomorphisms, their {\bf composition} $((\varphi\circ\psi)_0,(\varphi\circ\psi)_1,(\varphi\circ\psi)_2)$ is given by $(\varphi\circ\psi)_0=\varphi_0\circ\psi_0$,  $(\varphi\circ\psi)_1=\varphi_1\circ\psi_1$, and
 $$(\varphi\circ\psi)_2(x,y,z)=\psi_2(\varphi_0(x),\varphi_0(y),\varphi_0(z))+\psi_1(\varphi_2(x,y,z)).$$
  The {\bf identity homomorphism} $id_{\huaT}:\huaT\longrightarrow\huaT$ has the identity chain map as its underlying map, together with $(id_\huaT)_2=0$.

\begin{defi}\label{defi:Lwuqiong 2hom}
Let $\huaT$ and $\huaT'$ be $2$-term homotopy Lie triple systems, and $\varphi,\psi :\huaT \longrightarrow \huaT'$ be  homomorphisms. A  {\bf $2$-homomorphism} $\tau:\varphi\Rightarrow\psi$ is a chain homotopy such that for all $x_1,x_2,x_3\in \huaT_0 $, the following equation holds:
\begin{eqnarray}\label{2homoinfinity}
(\varphi_2-\psi_2)(x_1,x_2,x_3)&=&[\varphi_0(x_1),\varphi_0(x_2),\tau(x_3)]' +[\dM'\tau(x_1),\tau(x_2), \varphi_0(x_3)]'\\
\nonumber &&+[\dM'\tau(x_1),\dM'\tau(x_2),\tau(x_3)]'+c.p.-\tau([x_1,x_2,x_3]).
\end{eqnarray}
\end{defi}

Now we define the vertical and horizontal composition for these $2$-homomorphisms.
Let $\huaT$, $\huaT'$ be $2$-term homotopy Lie triple systems, and $\varphi,\psi,\mu:\huaT\longrightarrow \huaT'$ be homomorphisms. Let $\tau:\varphi\Rightarrow \psi$ and $\tau':\psi\Rightarrow \mu$ be $2$-homomorphisms.
The {\bf vertical composition} of $\tau$ and $\tau'$, denoted by $\tau'\tau$, is given by $\tau'\tau=\tau'+\tau$.

Let $\huaT$, $\huaT'$, $\huaT''$ be $2$-term homotopy Lie triple systems,   $\varphi,\psi,:\huaT\longrightarrow \huaT'$ and $\varphi',\psi':\huaT'\longrightarrow \huaT''$  be homomorphisms, and $\tau:\varphi\Rightarrow \psi$ and $\tau':\varphi'\Rightarrow \psi'$ be $2$-homomorphisms.
The {\bf horizontal composition} of $\tau$ and $\tau'$, denoted by $\tau'\circ\tau$, is given by $\tau'\circ\tau(x)=\tau'_{\varphi_0(x)}+\varphi_1'\tau(x)$.

Finally, given a homomorphism $\varphi$, the {\bf identity $2$-homomorphism} $1_{\varphi}:\varphi\Rightarrow \varphi$ is the zero chain homotopy  $1_{\varphi}(x)=0.$

It is straightforward to see that
\begin{pro}
 There is a  $2$-category \TTHL with $2$-term homotopy Lie triple systems as objects,  homomorphisms between $2$-term homotopy Lie triple systems as morphisms,   $2$-homomorphisms as $2$-morphisms.
\end{pro}

\subsection{Lie triple 2-systems}\label{lt2s}

In this subsection, we define Lie triple 2-systems, which are the categorification of Lie triple systems, and show that the 2-category of Lie triple 2-systems is equivalent to the 2-category of 2-term homotopy Lie triple systems.

\begin{defi}\label{defi:lie32sys}
A Lie triple $2$-system consists of:
\begin{itemize}
\item[$\bullet$] a $2$-vector spaces $L$;
\item[$\bullet$] a trilinear functor, the $\textbf{bracket}$, $\{\cdot,\cdot,\cdot\}: L\times L\times L\longrightarrow L$, such that for all $x_i\in L$, we have
\begin{eqnarray*}
 \{x_1,x_1,x_2\}&=&0,\\
 \{x_1,x_2,x_3\}+\{x_2,x_3,x_1\}+\{x_3,x_1,x_2\}&=&0;
\end{eqnarray*}
\item[$\bullet$] a multilinear natural isomorphism $\huaJ_{x_1,x_2,x_3,x_4,x_5}$ for all $x_i\in L_0$,
$$\{x_1,x_2,\{x_3,x_4,x_5\}\}\stackrel{\huaJ_{x_1,x_2,x_3,x_4,x_5}}{\longrightarrow }\{\{x_1,x_2,x_3\},x_4,x_5\}+\{x_3,\{x_1,x_2,x_4\},x_5\}+\{x_3,x_4,\{x_1,x_2,x_5\}\},$$
\end{itemize}
such that for all $x_1,\ldots,x_7\in L_0$, the following fundamental identity holds:
\begin{equation}\label{Jacobiator indentity}
\{x_1,x_2,\huaJ_{x_3,x_4,x_5,x_6,x_7}\}(\huaJ_{x_1,x_2,\{x_3,x_4,x_5\},x_6,x_7}+\huaJ_{x_1,x_2,x_5,\{x_3,x_4,x_6\},x_7}+\huaJ_{x_1,x_2,x_5,x_6,\{x_3,x_4,x_7\}})
\end{equation}
$$(\{x_5,x_6,\huaJ_{x_1,x_2,x_3,x_4,x_7}\}+1)(\{x_5,\huaJ_{x_1,x_2,x_3,x_4,x_6},x_7\}+\{\huaJ_{x_1,x_2,x_3,x_4,x_5},x_6,x_7\}+1)=$$
$$\huaJ_{x_1,x_2,x_3,x_4,\{x_5,x_6,x_7\}}(\{x_3,x_4,\huaJ_{x_1,x_2,x_5,x_6,x_7}\}+1)(\huaJ_{\{x_1,x_2,x_3\},x_4,x_5,x_6,x_7}+\huaJ_{x_3,\{x_1,x_2,x_4\},x_5,x_6,x_7}+$$
$$\huaJ_{x_3,x_4,\{x_1,x_2,x_5\},x_6,x_7}+\huaJ_{x_3,x_4,x_5,\{x_1,x_2,x_6\},x_7}+\huaJ_{x_3,x_4,x_5,x_6,\{x_1,x_2,x_7\}}),$$
or, in terms of a commutative diagram,
$$
\xymatrix{&\{x_1,x_2,\{x_3,x_4,\{x_5,x_6,x_7\}\}\}\ar[dr]^-{\huaJ_{x_1,x_2,x_3,x_4,\{x_5,x_6,x_7\}}}\ar[dl]_-{\{x_1,x_2,\huaJ_{x_3,x_4,x_5,x_6,x_7}\}}&\\
A \ar[d]_{\alpha}&& B\ar[dd]^{\beta}\\
C \ar[d]_{\gamma}&& \\
D \ar[dr]_{\delta}&& E \ar[dl]^{\eta}\\
&F}
$$
where $A,B,C,D,E,F$ and $\alpha,\beta,\gamma,\delta,\eta$ are given by
{\footnotesize
\begin{eqnarray*}
A&=&\{x_1,x_2,\{\{x_3,x_4,x_5\},x_6,x_7\}\}+\{x_1,x_2,\{x_5,\{x_3,x_4,x_6\},x_7\}\}+\{x_1,x_2,\{x_5,x_6,\{x_3,x_4,x_7\}\}\},\\
B&=&\{x_3,x_4,\{x_1,x_2,\{x_5,x_6,x_7\}\}\}+\{\{x_1,x_2,x_3\},x_4,\{x_5,x_6,x_7\}\}+\{x_3,\{x_1,x_2,x_4\},\{x_5,x_6,x_7\}\},\\
C&=&\{\{x_1,x_2,\{x_3,x_4,x_5\}\},x_6,x_7\}+\{\{x_3,x_4,x_5\},\{x_1,x_2,x_6\},x_7\}+\{\{x_3,x_4,x_5\},x_6,\{x_1,x_2,x_7\}\}\\
&&+\{\{x_1,x_2,x_5\},\{x_3,x_4,x_6\},x_7\}+\{x_5,\{x_1,x_2,\{x_3,x_4,x_6\}\},x_7\}+\{x_5,\{x_3,x_4,x_6\},\{x_1,x_2,x_7\}\}\\
&&+\{\{x_1,x_2,x_5\},x_6,\{x_3,x_4,x_7\}\}+\{x_5,\{x_1,x_2,x_6\},\{x_3,x_4,x_7\}\}+\{x_5,x_6,\{x_1,x_2,\{x_3,x_4,x_7\}\}\},\\
D&=&\{\{x_1,x_2,\{x_3,x_4,x_5\}\},x_6,x_7\}+\{\{x_3,x_4,x_5\},\{x_1,x_2,x_6\},x_7\}+\{\{x_3,x_4,x_5\},x_6,\{x_1,x_2,x_7\}\}\\
&&+\{\{x_1,x_2,x_5\},\{x_3,x_4,x_6\},x_7\}+\{x_5,\{x_1,x_2,\{x_3,x_4,x_6\}\},x_7\}+\{x_5,\{x_3,x_4,x_6\},\{x_1,x_2,x_7\}\}\\
&&+\{\{x_1,x_2,x_5\},x_6,\{x_3,x_4,x_7\}\}+\{x_5,\{x_1,x_2,x_6\},\{x_3,x_4,x_7\}\}+\{x_5,x_6,\{\{x_1,x_2,x_3\},x_4,x_7\}\}\\
&&+\{x_5,x_6,\{x_3,\{x_1,x_2,x_4\},x_7\}\}+\{x_5,x_6,\{x_3,x_4,\{x_1,x_2,x_7\}\}\},\\
E&=&\{\{x_1,x_2,x_3\},x_4,\{x_5,x_6,x_7\}\}+\{x_3,\{x_1,x_2,x_4\},\{x_5,x_6,x_7\}\}+\{x_3,x_4,\{\{x_1,x_2,x_5\},x_6,x_7\}\}\\
&&+\{x_3,x_4,\{x_5,\{x_1,x_2,x_6\},x_7\}\}+\{x_3,x_4,\{x_5,x_6,\{x_1,x_2,x_7\}\}\},\\
F&=&\{\{\{x_1,x_2,x_3\},x_4,x_5\},x_6,x_7\}+\{\{x_3,\{x_1,x_2,x_4\},x_5\},x_6,x_7\}+\{\{x_3,x_4,\{x_1,x_2,x_5\}\},x_6,x_7\}\\
&&+\{x_5,\{\{x_1,x_2,x_3\},x_4,x_6\},x_7\}+\{x_5,\{x_3,\{x_1,x_2,x_4\},x_6\},x_7\}+\{x_5,\{x_3,x_4,\{x_1,x_2,x_6\}\},x_7\}\\
&&+\{\{x_3,x_4,x_5\},\{x_1,x_2,x_6\},x_7\}+\{\{x_3,x_4,x_5\},x_6,\{x_1,x_2,x_7\}\}+\{\{x_1,x_2,x_5\},\{x_3,x_4,x_6\},x_7\}\\
&&+\{x_5,\{x_3,x_4,x_6\},\{x_1,x_2,x_7\}\}+\{\{x_1,x_2,x_5\},x_6,\{x_3,x_4,x_7\}\}+\{x_5,\{x_1,x_2,x_6\},\{x_3,x_4,x_7\}\}\\
&&+\{x_5,x_6,\{\{x_1,x_2,x_3\},x_4,x_7\}\}+\{x_5,x_6,\{x_3,\{x_1,x_2,x_4\},x_7\}\}+\{x_5,x_6,\{x_3,x_4,\{x_1,x_2,x_7\}\}\},\\
\alpha&=&\huaJ_{x_1,x_2,\{x_3,x_4,x_5\},x_6,x_7}+\huaJ_{x_1,x_2,x_5,\{x_3,x_4,x_6\},x_7}+\huaJ_{x_1,x_2,x_5,x_6,\{x_3,x_4,x_7\}},\\
\beta&=&\{x_3,x_4,\huaJ_{x_1,x_2,x_5,x_6,x_7}\}+1,\\
\gamma&=&\{x_5,x_6,\huaJ_{x_1,x_2,x_3,x_4,x_7}\}+1,\\
\delta&=&\{x_5,\huaJ_{x_1,x_2,x_3,x_4,x_6},x_7\}+\{\huaJ_{x_1,x_2,x_3,x_4,x_5},x_6,x_7\}+1,\\
\eta&=&\huaJ_{\{x_1,x_2,x_3\},x_4,x_5,x_6,x_7}+\huaJ_{x_3,\{x_1,x_2,x_4\},x_5,x_6,x_7}+\huaJ_{x_3,x_4,\{x_1,x_2,x_5\},x_6,x_7}\\
&&+\huaJ_{x_3,x_4,x_5,\{x_1,x_2,x_6\},x_7}+\huaJ_{x_3,x_4,x_5,x_6,\{x_1,x_2,x_7\}}.
\end{eqnarray*}
}
\end{defi}

We continue by setting up a $2$-category of Lie triple $2$-systems.

\begin{defi}\label{defi:lthomo}
  Given Lie triple $2$-systems $L$ and $L'$, a {\bf homomorphism} $F:L\longrightarrow L'$ consists of:
\begin{itemize}
\item[$\bullet$] A linear functor $F$ from the underlying $2$-vector space of $L$ to that of $L',$ and

\item[$\bullet$] a trilinear natural transformation
 $$F_2(x,y,z):\{F_0(x),F_0(y),F_0(z)\}'\longrightarrow F_0\{x,y,z\}$$
 such that for all $x, y, z, x_i \in L_0$, we have
\begin{eqnarray*}
 F_2(x,x,y)&=&0,\\
 F_2(x,y,z)&=&-F_2(y,z,x)-F_2(z,x,y);
\end{eqnarray*}
 and the following diagram commutes:
$$
\xymatrix{\{F_0(x_1),F_0(x_2),\{F_0(x_3),F_0(x_4),F_0(x_5)\}'\}'\ar[d]^{\{1,1,F_2\}}\ar[rrrr]^{\qquad \qquad \qquad \qquad \qquad \qquad \qquad \qquad \qquad \huaJ'_{F_0(x_1),F_0(x_2),F_0(x_3),F_0(x_4),F_0(x_5)}\qquad \qquad \qquad \qquad \qquad}&&&& A \ar[d]_{\{F_2,1,1\}+\{1,F_2,1\}+\{1,1,F_2\}}\\
\{F_0(x_1),F_0(x_2),F_0\{x_3,x_4,x_5\}\}'\ar[d]^{F_2}&&&&
B \ar[d]_{F_2+F_2+F_2}\\
F_0\{x_1,x_2,\{x_3,x_4,x_5\}\}\ar[rrrr]^{\qquad F_1(\huaJ_{x_1,x_2,x_3,x_4,x_5})}&&&& C,}
$$
where $A,B,C$ are given by
\begin{eqnarray*}
A&=&\{\{F_0(x_1),F_0(x_2),F_0(x_3)\}',F_0(x_4),F_0(x_5)\}'+\{F_0(x_3),\{F_0(x_1),F_0(x_2),F_0(x_4)\}',F_0(x_5)\}'\\
&&+\{F_0(x_3),F_0(x_4),\{F_0(x_1),F_0(x_2),F_0(x_5)\}'\}',\\
B&=&\{F_0\{x_1,x_2,x_3\},F_0(x_4),F_0(x_5)\}'+\{F_0(x_3),F_0\{x_1,x_2,x_4\},F_0(x_5)\}'\\
&&+\{F_0(x_3),F_0(x_4),F_0\{x_1,x_2,x_5\}\}',\\
C&=&F_0\{\{x_1,x_2,x_3\},x_4,x_5\}+F_0\{x_3,\{x_1,x_2,x_4\},x_5\}+F_0\{x_3,x_4,\{x_1,x_2,x_5\}\}.
\end{eqnarray*}
\end{itemize}
\end{defi}

The identity homomorphism $\Id_L:L\longrightarrow L$ has the identity functor as its underlying functor, together with an identity natural transformation as $(\Id_L)_2.$ Let $L,L'$ and $L''$ be Lie triple $2$-systems, the composite of homomorphisms $F:L\longrightarrow L'$ and $G:L'\longrightarrow L''$ which we denote by $G\circ F$, is given by letting the functor $((G\circ F)_0,(G\circ F)_1)$ be the usual composition of $(G_0,G_1)$ and $(F_0,F_1)$, and letting $(G\circ F)_2$ be the following composite:

$$\xymatrix{
  \{G_0\circ F_0(x),G_0\circ F_0(y),G_0\circ F_0(z)\}'' \ar[dd]_{G_2(F_0(x),F_0(y),F_0(z))} \ar[dr]^{\ \ \ \ (G\circ F)_2(x,y,z)} \\& G_0\circ F_0\{x,y,z\} .  \\
  G_0\{F_0(x),F_0(y),F_0(z)\}'  \ar[ur]_{G_1(F_2(x,y,z))}                     }
$$

We also have $2$-homomorphisms between homomorphisms:

\begin{defi}\label{defi:liet2homo}
Let $F,G:L\longrightarrow L'$ be homomorphisms. A {\bf $2$-homomorphism} $\theta:F\Rightarrow G$ is a linear natural transformation from $F$ to $G$ such that the following diagram commutes:
$$
\xymatrix{\{F_0(x),F_0(y),F_0(z)\}'\ar[d]^{\{\theta_x,\theta_y,\theta_z\}'}\ar[rr]^-{F_2}&&F_0\{x,y,z\}\ar[d]_{\theta_{\{x,y,z\}}}\\
\{G_0(x),G_0(y),G_0(z)\}'\ar[rr]^-{G_2}&&
G_0\{x,y,z\}.}
$$
\end{defi}

Since $2$-homomorphisms are just natural transformations with an extra property, we vertically and horizontally compose these in the usual way, and an identity $2$-homomorphism is just an identity natural transformation.

It is straightforward to see that

\begin{pro}
There is a   $2$-category   \LTTS with   Lie triple $2$-systems as objects, homomorphisms  as morphisms, and $2$-homomorphisms as $2$-morphisms.
\end{pro}

\subsection{The equivalence between 2-term homotopy Lie triple systems and Lie triple $2$-systems}

Now we  establish the equivalence between the 2-category of Lie triple $2$-systems and that of $2$-term homotopy Lie triple systems. This result is based on the equivalence between $2$-vector spaces and $2$-term chain complexes described in \cite{baez:2algebras}.

\begin{thm}\label{thm:equivalence}
The $2$-categories \TTHL and \LTTS are $2$-equivalent.
\end{thm}
\begin{proof}
First we construct a 2-functor $T:$ \TTHL $\longrightarrow$ \LTTS. Given a $2$-term homotopy Lie triple system $\huaT=(\huaT_0,\huaT_{-1},\dM,[\cdot,\cdot,\cdot],J)$, we have a $2$-vector space $L=(L_0,L_1,s,t,i,\circ)$ with $L_0=\huaT_0, L_1=\huaT_0\oplus \huaT_{-1}$, and the source and the target map are given by $s(x+f)=x$ and $t(x+f)=x+\dM f$. Define
  a trilinear functor $\{\cdot,\cdot,\cdot\}:L\times L\times L\longrightarrow L$ by
  \begin{eqnarray*}
  \{x+f,y+g,z+h\}&=&[x,y,z]+[x,y,h]+[x,g,z]+[f,y,z]\\
  &&+[\dM f,g,z]+[\dM f,y,h]+[x,\dM g,h]+[\dM f,\dM g,h],
  \end{eqnarray*}
and define the fundamentor $\huaJ_{x_1,x_2,x_3,x_4,x_5}$ by
$$\huaJ_{x_1,x_2,x_3,x_4,x_5}=([x_1,x_2,[x_3,x_4,x_5]],J(x_1,x_2,x_3,x_4,x_5)).$$
The source of $\huaJ_{x_1,x_2,x_3,x_4,x_5}$ is $[x_1,x_2,[x_3,x_4,x_5]]$ as desired. By (h) in the Definition \ref{2termHomLTS}, its target is
\begin{eqnarray*}
t(\huaJ_{x_1,x_2,x_3,x_4,x_5})&=&[x_1,x_2,[x_3,x_4,x_5]]+\dM J(x_1,x_2,x_3,x_4,x_5)\\
&=&[[x_1,x_2,x_3],x_4,x_5]+[x_3,[x_1,x_2,x_4],x_5]+[x_3,x_4,[x_1,x_2,x_5]],
\end{eqnarray*}
 as desired.
 By Conditions (i), (j) and (k) in Definition \ref{2termHomLTS}, we deduce that $\huaJ_{x_1,x_2,x_3,x_4,x_5}$ is a natural transformation.
By Condition (l) in Definition \ref{2termHomLTS}, we can deduce that the fundamentor identity holds.
This completes the construction of a Lie triple $2$-system $L=T(\huaT)$ from a $2$-term homotopy Lie triple system $\huaT.$

We go on to construct a Lie triple $2$-system homomorphism $T(\phi):T(\huaT)\longrightarrow T(\huaT')$ from a homotopy Lie triple system homomorphism $\phi=(\phi_0,\phi_1,\phi_2):\huaT\longrightarrow \huaT'$ between $2$-term homotopy Lie triple systems.
Let $T(\huaT)=L$ and $T(\huaT')=L'.$ We define the underlying linear functor of $T(\phi)=F$ with $F_0=\phi_0$, $F_1=\phi_0\oplus\phi_1.$ Define $F_2:\huaT_0\times \huaT_0\times \huaT_0\longrightarrow \huaT_0'\oplus \huaT_{-1}'$ by
$$F_2(x_1,x_2,x_3)=([\phi_0(x_1),\phi_0(x_2),\phi_0(x_3)]',\phi_2(x_1,x_2,x_3)).$$
Then $F_2(x_1,x_2,x_3)$ is a natural isomorphism from $[F_0(x_1),F_0(x_2),F_0(x_3)]'$ to $F_0[x_1,x_2,x_3]$, and $F=(F_0,F_1,F_2)$ is a homomorphism from $L$ to $L'$.
We can also prove that $T$ preserve identities and composition of homomorphisms. So $T$ is a functor.

Furthermore, to construct $T$ to be a 2-functor, we only need to define $T$ on 2-morphisms. Let $\varphi,\psi:\huaT\longrightarrow\huaT'$ be homomorphisms and $\tau:\varphi\Rightarrow\psi$ a $2$-homomorphism. Then we define  $$T(\tau)(x)=(\varphi_0(x),\tau(x)).$$ By \eqref{2homoinfinity}, $T(\tau)$ is a $2$-homomorphism from $T(\phi)$ to $T(\psi)$. It is obvious that $T$ preserves the compositions and identities.
Thus, $T$ is a 2-functor from \TTHL to \LTTS.

Next we   construct a 2-functor $S:$\LTTS $\longrightarrow$ \TTHL.
Given a Lie triple $2$-system $L$, we obtain a complex of vector spaces  $\huaT=S(L)$ with $\huaT_0=L_0,\huaT_{-1}=\ker(s)$ and $\dM=t|_{\ker(s)}$.
For all $x_1,x_2,x_3,x_4,x_5\in \huaT_0=L_0$ and $f,g,h\in \huaT_{-1}\subseteq L_1$,
we define  $[\cdot,\cdot,\cdot]$ and $J$ as follows:
\begin{itemize}
\item[\rm(i)~]$[x_1,x_2,x_3]=\{1_{x_1},1_{x_2},1_{x_3}\},$
\item[\rm(ii)~]$[x_1,x_2,h]=\{1_{x_1},1_{x_2},h\},~[x_1,h,x_2]=\{1_{x_1},h,1_{x_2}\},$
\item[\rm(iii)]$[x_1,f,g]=0,~[f,g,x_1]=0,~[f,g,h]=0,$
\item[\rm(iv)]$J(x_1,x_2,x_3,x_4,x_5)=\huaJ_{x_1,x_2,x_3,x_4,x_5}-1_{s(\huaJ_{x_1,x_2,x_3,x_4,x_5})}.$
\end{itemize}

The various conditions of $L$ being a Lie triple $2$-system imply that $\huaT=(\huaT_0,\huaT_{-1},\dM,[\cdot,\cdot,\cdot],J)$ is $2$-term homotopy Lie triple system. This completes the construction of a $2$-term homotopy Lie triple system $\huaT=S(L)$ from a Lie triple $2$-system $L$.

Let $L$ and $L'$ be Lie triple $2$-systems, and $F=(F_0,F_1,F_2):L\longrightarrow L'$ a homomorphism. Let $S(L)=\huaT$ and $S(L')=\huaT'$. We go on to  construct a homotopy Lie triple system homomorphism $\phi=S(F):\huaT\longrightarrow \huaT'$.
Let $\phi_0=F_0$, $\phi_1=F_1|_{\ker(s)}$. Define $\phi_2:\huaT_0\times \huaT_0\times \huaT_0\longrightarrow \huaT'_{-1}$ by
$$\phi_2(x_1,x_2,x_3)=F_2(x_1,x_2,x_3)-1_{s(F_2(x_1,x_2,x_3))}.$$
Then $\phi_2$ satisfies
\begin{eqnarray*}
  \phi_2(x_1,x_1,x_2)&=&0,\\
  \phi_2(x_1,x_2,x_3)&=&-\phi_2(x_2,x_3,x_1)-\phi_2(x_3,x_1,x_2),
\end{eqnarray*}
and
\begin{eqnarray*}
\dM' \phi_2(x_1,x_2,x_3)&=&(t'-s')F_2(x_1,x_2,x_3)\\
&=&\phi_0([x_1,x_2,x_3])-[\phi_0(x_1),\phi_0(x_2),\phi_0(x_3)]'.
\end{eqnarray*}
 The naturality of $F_2$ gives  \eqref{eq:homo4} and \eqref{eq:homo5} in Definition \ref{defi:Lwuqiong hom}, and the fundamentor identity gives  \eqref{eq:homo6} in Definition \ref{defi:Lwuqiong hom}. Thus, $\phi=S(F)$ is a homomorphism between 2-term homotopy Lie triple systems.

Let $F,G:L\longrightarrow L'$ be Lie triple $2$-system homomorphisms and  $\theta:F\Rightarrow G$  a $2$-homomorphism.
Let $\varphi=S(F),\psi=S(G) :\huaT \longrightarrow \huaT'$ be the corresponding homotopy Lie triple system homomorphisms. We define
$$S(\theta)(x)=\theta(x)-1_{s'(\theta(x))}.$$
 By the commutative diagram in Definition \ref{defi:liet2homo}, we can deduce that \eqref{2homoinfinity} holds. Thus, $S(\theta)$  is a $2$-homomorphism. It is straightforward to deduce that $S$ preserves the compositions and identities. Thus  $S$ is a 2-functor from \LTTS to \TTHL.

Now, we construct natural isomorphisms $\alpha:ST\Rightarrow 1_{\rm\bf LieTri2Sys}$ and $\beta:TS\Rightarrow 1_{\rm\bf 2TermHomLTS}$. To construct $\alpha$, consider the $2$-term chain complex of $S(L)$
$$\xymatrix{ \ker(s) \ar[rr]^{t|_{\ker(s)}} && L_0}.$$
Applying $T$ to this result, we obtain a $2$-vector space $L'$ with the space $L_0$ of objects and the space $L_0\oplus\ker(s)$ of morphisms. The source map for this $2$-vector space is given by $s'(x+f)=x$, the target map is given by $t'(x+f)=x+t(f)$. We thus can define an isomorphism $\alpha_L:L'\longrightarrow L$ by setting
\begin{eqnarray*}
  (\alpha_L)_0(x)&=&x,\\
  (\alpha_L)_1(x+f)&=&i(x)+f.
\end{eqnarray*}
It is easy to check that $\alpha_L$ is a linear functor. It is an isomorphism thanks to the fact that every morphism in $L$ can be uniquely written as $i(x)+f$ where $x$ is an object and $f\in \ker(s)$.

To construct $\beta$, consider a $2$-term chain complex $\huaT$ given by
$$\xymatrix{ \huaT_{-1} \ar[rr]^{d} && \huaT_0}.$$
Then $T(\huaT)$ is the $2$-vector space with the space $\huaT_0$ of objects, the space $\huaT_0\oplus \huaT_{-1}$ of morphisms, together with the source and target maps $s(x+f)=x$ and $t(x+f)=x+\dM f$. Applying the functor $S$ to this $2$-vector space we obtain a $2$-term chain complex $\huaT'$ given by
$$\xymatrix{ \ker(s) \ar[rr]^{t|_{\ker(s)}} && \huaT_0}.$$
Since $\ker(s)=\{x+f|x=0\}\subseteq \huaT_0\oplus \huaT_{-1}$, there is an obvious isomorphism $\ker(s)\cong \huaT_{-1}$. Using this we obtain an isomorphism $\beta_\huaT:\huaT'\longrightarrow \huaT$ given by
$$\xymatrix{
 \ker(s) \ar[rr]^{t|_{\ker(s)}}\ar[d]_{\cong} && \huaT_0 \ar[d]^{1} \\
 \huaT_{-1} \ar[rr]^{d} && \huaT_0
}$$
where the square commutes because of how we have defined $t$. It is easy to verify that $\alpha$ and $\beta$ are indeed natural isomorphisms. We omit details. The proof is completed.

%In the end, it is easy to construct the natural isomorphisms $\alpha:ST\Rightarrow 1_{\rm\bf LieTri2Sys}$ and $\beta:TS\Rightarrow 1_{\rm\bf 2TermHomLTS}$. We omit details. The proof is completed.
\end{proof}

\section{Skeletal and strict Lie triple $2$-systems}\label{sec:ske}

By Theorem \ref{thm:equivalence}, we see that Lie triple $2$-systems and $2$-term homotopy Lie triple systems are equivalent. Thus, we will call a $2$-term homotopy Lie triple system a Lie triple $2$-system in the sequel.

A Lie triple $2$-system $(\huaT_0,\huaT_{-1},\dM,[\cdot,\cdot,\cdot],J)
$ is called {\bf skeletal} ({\bf strict})  if $\dM=0$ ($J=0$).

In this section, first we classify skeletal Lie triple $2$-systems via the third cohomology group. Then, we introduce the notion of a crossed module of Lie triple systems, and show that they are equivalent to strict Lie triple $2$-systems. %First we recall representations and cohomologies of Lie triple systems. At the end, we construct a 3-Lie 2-algebra from a 3-Lie algebra with a symplectic structure.

\begin{thm}
There is a one-to-one correspondence between skeletal Lie triple $2$-systems and quadruples $((\g,[\cdot,\cdot,\cdot]_\g),V,\theta,\omega)$, where $(\g,[\cdot,\cdot,\cdot]_\g)$ is a Lie triple system, $V$ is a vector space, $\theta$ is a representation of $\g$ on $V,$ and $\omega$ is a $3$-cocycle on $\g$ with values in $V$.
\end{thm}

\begin{proof}
Let $(\huaT_0,\huaT_{-1},\dM=0,[\cdot,\cdot,\cdot],J)$ be a skeletal Lie triple $2$-systems. By (h) in Definition \ref{2termHomLTS}, we see that $[\cdot,\cdot,\cdot]|_{\huaT_0}$ satisfies the fundamental identity. Thus, $(\huaT_0,[\cdot,\cdot,\cdot]|_{\huaT_0})$ is a Lie triple system.  $[\cdot,\cdot,\cdot]$ also gives rise to a map $\theta:\otimes^2\huaT_0\longrightarrow \huaT_{-1}$ by
 \begin{equation}
 \theta(x_1, x_2)(f)=[f,x_1,x_2].
 \end{equation}
   By (i), (j) and (k) in Definition \ref{2termHomLTS}, we deduce that $\theta$ is a representation of the Lie triple system $(\huaT_0,[\cdot,\cdot,\cdot]|_{\huaT_0})$ on $\huaT_{-1}$. Finally, by (l) in Definition \ref{2termHomLTS}, we get that $J$ is a 3-cocycle.

  Conversely, given a Lie triple system $(\g,[\cdot,\cdot,\cdot]_\g),$ a representation $\theta$ of $\g$ on a  vector space $V,$ and a $3$-cocycle $\omega$ on $\g$ with values in $V$, we define $\huaT_0=\g,$ $\huaT_{-1}=V$,  $\dM=0,$ and $[\cdot,\cdot,\cdot]$, $J$   by
\begin{eqnarray*}
  [x,y,z]&=&[x,y,z]_\g,\\
  {[f,x,y]}&=&\theta(x,y)(f), \\
  {[x,f,y]}&=&-\theta(x,y)(f), \\
  {[x,y,f]}&=&\theta(y,x)(f)-\theta(x,y)(f), \\
  J&=&\omega.
\end{eqnarray*}
 Then it is straightforward to deduce that $(\huaT_0,\huaT_{-1},\dM=0,[\cdot,\cdot,\cdot],J)$ is a skeletal Lie triple $2$-system. We omit details.
\end{proof}

Now we introduce the notion of a crossed module of Lie triple systems and show that they are equivalent to strict Lie triple $2$-systems.
\begin{defi}
A {\bf crossed module of Lie triple systems} is a quadruple  $((\frkg,[\cdot,\cdot,\cdot]_{\frkg}),(\frkh,[\cdot,\cdot,\cdot]_{\frkh}),\mu,\theta)$, where $(\frkg,[\cdot,\cdot,\cdot]_{\frkg})$ and $(\frkh,[\cdot,\cdot,\cdot]_{\frkh})$ are Lie triple systems, $\mu:\frkg\longrightarrow\frkh$  is a homomorphism of Lie triple systems, and   $\theta:\otimes^2\frkh\longrightarrow\Der(\frkg)$ is a representation of $\h$ on $\g$, such that for all $x,y,z\in \frkh,f,g,h\in\frkg,$ the following equalities hold:
\begin{eqnarray}
  \label{eq:cmc1}\mu(\theta(x,y)(f))&=&[\mu(f),x,y]_{\frkh},\\
  \label{eq:cmc2}\theta(\mu(f),\mu(g))(h)&=&[h,f,g]_{\frkg},\\
  \label{eq:cmc3}\mu(\theta(x,\mu(f))(g))&=&[\mu(g),x,\mu(f)]_{\frkh},\\
  \label{eq:cmc4}\mu(\theta(\mu(f),x)(g))&=&[\mu(g),\mu(f),x]_{\frkh}.
\end{eqnarray}
\end{defi}

\begin{rmk}
  The more general notion of crossed modules of $n$-Leibniz algebras had been given in \cite{CasasCMn-Lie}, and the relation with the third cohomology group is established there.
\end{rmk}

\begin{thm}
There is a one-to-one correspondence between strict Lie triple $2$-systems and crossed modules of Lie triple systems.
\end{thm}

\begin{proof}
Let $(\huaT_0,\huaT_{-1},\dM,[\cdot,\cdot,\cdot],J=0)$ be a strict Lie triple $2$-system. Define $\frkg=\huaT_{-1},\frkh=\huaT_0,$ and the following two bracket operations on $\frkg$ and $\frkh$:
\begin{eqnarray}
~[f,g,h]_{\frkg}&=&[\dM f,\dM g,h]=[\dM f,g,\dM h]=[f,\dM g,\dM h],\\
~[x,y,z]_{\frkh}&=&[x,y,z].
\end{eqnarray}
It is straightforward to see that both $(\frkg,[\cdot,\cdot,\cdot]_{\frkg})$ and $(\frkh,[\cdot,\cdot,\cdot]_{\frkh})$ are Lie triple systems. Let $\mu=\dM,$ by Condition (a) in Definition \ref{2termHomLTS}, we have
$$\mu[f,g,h]_{\frkg}=\dM [\dM f,\dM g,h]=[\dM f,\dM g,\dM h]=[\mu(f),\mu(g),\mu(h)]_{\frkh},$$
which implies that $\mu$ is a homomorphism of Lie triple systems. Define $\theta:\otimes^2\frkh\longrightarrow\Der(\frkg)$ by
$$\theta(x,y)(f)=[f,x,y].$$
By Conditions (i), (j) and (k) in Definition \ref{2termHomLTS}, we can obtain that $\theta$ is a representation. Furthermore, we have
\begin{eqnarray*}
  \mu(\theta(x,y)(f))&=&\dM (\theta(x,y)(f))=\dM [f,x,y]=[\dM f,x,y]=[\mu(f),x,y]_{\frkh},\\
  \theta(\mu(f),\mu(g))(h)&=&[h,\dM f,\dM g]=[h,f,g]_{\frkg},\\
  \mu(\theta(x,\mu(f))(g))&=&\dM (\theta(x,\mu(f))(g))=\dM [g,x,\mu(f)]=[\dM g,x,\mu(f)]=[\mu(g),x,\mu(f)]_{\frkh},\\
  \mu(\theta(\mu(f),x)(g))&=&\dM (\theta(\mu(f),x)(g))=\dM [g,\mu(f),x]=[\dM g,\mu(f),x]=[\mu(g),\mu(f),x]_{\frkh}.
\end{eqnarray*}
Therefore we obtain a crossed module of Lie triple systems.

Conversely, a crossed module of Lie triple systems gives rise to a strict Lie triple $2$-system, in which $\huaT_{-1}=\frkg,$  $\huaT_0=\frkh$, $\dM=\mu$, and the trilinear map $[\cdot,\cdot,\cdot]$ is given by
\begin{eqnarray*}
  [x,y,z]&=&[x,y,z]_{\frkh},\\
  {[f,x,y]}&=&\theta(x,y)(f),\\
  {[x,f,y]}&=&-\theta(x,y)(f), \\
  {[x,y,f]}&=&\theta(y,x)(f)-\theta(x,y)(f).
\end{eqnarray*}
The crossed module conditions give various conditions for strict Lie triple $2$-systems. We omit details. The proof is completed.
\end{proof}

 \end{document}